\documentclass[10pt,a4paper]{article}

\usepackage[latin1]{inputenc}
\usepackage[left=2.00cm, right=2.00cm, top=2.00cm]{geometry}
\usepackage{amsmath}
\usepackage{amsthm}
\usepackage{amsfonts}
\usepackage{amssymb}
\usepackage{graphicx}
\usepackage{tikz-cd}
\usepackage{subcaption}

\theoremstyle{plain}
\newtheorem{thm}{Theorem}[section]

\newtheorem{propn}[thm]{Proposition}
\newtheorem{lem}[thm]{Lemma}
\theoremstyle{definition}

\numberwithin{equation}{section}

\newcommand{\C}{\mathbb{C}}
\newcommand{\R}{\mathbb{R}}
\newcommand{\Q}{\mathbb{H}}
\newcommand{\Span}{\text{Span}}
\DeclareMathOperator{\Imag}{\text{Im}}
\DeclareMathOperator{\ad}{ad}
\newcommand{\orb}{\mathcal{O}}

\title{Singular Reduction of the 2-Body Problem on the 3-Sphere and the 4-Dimensional Spinning Top}
\author{Philip Arathoon}
\date{March 2019}
\begin{document}
	\maketitle
	\begin{abstract}
		We consider the dynamics and symplectic reduction of the 2-body problem on a sphere of arbitrary dimension. It suffices to consider the case for when the sphere is 3-dimensional and where we take the group of symmetries to be \(SO(4)\). As the 3-sphere is a group, both left and right multiplication on itself are commuting symmetries which together generate the full symmetry group. This gives rise to a notion of left and right momenta for the problem, and allows for a reduction in stages, first by the left and then the right, or vice versa. The intermediate reduced spaces obtained by left or right reduction are shown to be coadjoint orbits of the special Euclidean group \(SE(4)\). The full reduced spaces are generically 4-dimensional and we describe these spaces and their singular strata.
		
		The dynamics of the 2-body problem descend through a double cover to give a dynamical system on \(SO(4)\), which after reduction is the same as that of a 4-dimensional spinning top with symmetry. This connection allows us to `hit two birds with one stone' and derive results about both the spinning top and the 2-body problem simultaneously. We provide the equations of motion on the reduced spaces and fully classify the relative equilibria and discuss their stability.
		
	\end{abstract}

\section*{Background and outline of the paper}
The 2-body problem in ordinary flat Euclidean space enjoys not only the symmetries of rotation and translation, but also the larger group of Galilean transformations. It is through such a transformation into a centre of mass frame that the problem reduces to the ordinary Kepler problem. For quite some time now people have been interested in the generalisation of the 2-body problem to spaces of constant non-zero curvature, where it is no longer the case that it reduces to the problem of one body. Principal contributions in this area include, but are not limited to, the numerous works of Diacu, see in particular the book \cite{diacu}, the papers of Borisov, Mamaev and others in \cite{borisov2004,borisovrestricted,borisov2016}, and the work of Shchepetilov in \cite{badshchepetilov}. Recently the case for the 2-dimensional surfaces of constant curvature, the sphere and the hyperbolic/Lobachevsky plane, has been comprehensively treated in \cite{james}. Therein they perform symplectic reduction on the subset where the action is free, and fully classify the relative equilibria. This paper is written in response to open problems presented at the end of that work, in particular we aim to address the generalisation of their results to the 3-sphere.

Typically, when discussing the 2-body problem in spaces with non-zero curvature one begins by highlighting that, unlike in the Euclidean case, the symmetry of the problem no longer includes translations. This is indeed true for all but the two cases for when the sphere is itself a group; that is, for when it is the circle or the 3-sphere. Consequently, for this very special case translations do exist, and in a sense there are more than in the flat case. As the group is non-abelian, there is a difference between translations given by group multiplication on the left and right. The entire group \(SO(4)\) of symmetries is generated in this way. This establishes the well-known double cover of \(SO(4)\) and allows us to identify both the configuration space for the 2-body problem on the 3-sphere with its group of symmetries. It is this curious aspect of the problem which underlies the work contained in this paper, an outline of which we now provide.

We begin by casting the problem entirely in terms of quaternions. This is a natural setting for the 3-sphere which may be taken to be the set of unit-length quaternions, and for where group multiplication is simply given by multiplication of quaternions. As left and right multiplication commute, we set out a plan to reduce the problem in stages: first reducing by either the left or right translations to obtain an intermediate reduced space, and then by the other to obtain the full reduced space. We conclude the introduction by demonstrating how the dynamics project under the \(SO(4)\) double cover to give the symmetric heavy top in 4-dimensional space.

By drawing an analogy with the reduction of the Lagrange top, where reduction is also done in stages by first reducing in the body frame and then in the space frame about the axes of symmetry, we invoke the Semidirect Product Reduction by Stages Theorem to express the left and right reduced spaces as coadjoint orbits of \(SE(4)\). This is entirely analogous to the situation for the Lagrange top, whose intermediate reduced spaces are coadjoint orbits of \(SE(3)\). As the actions of left and right translation are free, these reduced spaces are well-behaved smooth manifolds. However, to complete the full reduction by the residual left or right action we necessarily have to handle non-free and singular points of the momentum. We employ the methods of singular and universal reduction through the use of some invariant theory to describe these reduced spaces, which are generically 4-dimensional. We give the corresponding equations of motion on the full reduced space for both the 2-body problem and the spinning top, and explicitly exhibit an additional integral for the symmetric spinning top demonstrating complete integrability.

We then turn our attention to the relative equilibria. Instead of classifying these by finding fixed points in the reduced space directly, we instead find solutions in the intermediate left and right reduced spaces which are the orbits of one-parameter subgroups. This pleasantly turns out to be comparatively easy, amounting to an entirely linear problem in Euclidean geometry. Having classified the solutions in the left reduced space, it is then only a matter of reconstruction to obtain the full classification of relative equilibria on the original space. We then explore the stability of the corresponding fixed points in the full reduced space by linearising the flow at these points. In this way, for the hamiltonians corresponding to the 2-body problem and the Lagrange top, we derive the linear stability results for the relative equilibria. We also provide the images of the energy-momentum map, and in doing so, obtain a picture for the bifurcations of the relative equilibria. Finally, in an effort to strengthen the stability results, we give the signature of the Hessian at the fixed points for the relative equilibria for the 2-body problem, and obtain the strongest possible stability result, that of Lyapunov stability, for linearly stable points of the Lagrange top.
\section{Introduction}
\subsection{The problem setting}
Consider the motion of two interacting particles of mass \(m_1\) and \(m_2\) constrained to move on the unit sphere \(S^n\subset\R^{n+1}\). The interaction is governed by a potential \(V\) which is a function of the distance between the two particle positions, where \(\R^{n+1}\) is equipped with the standard Euclidean metric. 

The initial position and velocity vectors of the two particles span at most a 4-dimensional linear subspace (for \(n>2\)). The intersection of this with the sphere is an equatorial 3-sphere. Reflection in this subspace is a symmetry of the dynamics and therefore, the motion must be forever contained to this 3-sphere. Consequently, it suffices to consider the case \(n=3\). This case also encompasses those for \(n=1\) and \(n=2\).

The space \(\R^4\) may be identified with the algebra \(\Q=\Span\{1,i,j,k\}\) of real quaternions. The standard inner product is written in terms of quaternionic multiplication by
\begin{equation}\label{innerproduct}
\langle p,q\rangle=\frac{1}{2}\left(pq^\dagger+qp^\dagger\right)
\end{equation}
where \(q^\dagger\) denotes the complex conjugate of the quaternion \(q\). We will denote the unit sphere by the letter \(G\) to highlight that it forms a group with respect to quaternionic multiplication. The hamiltonian formulation of the problem has phase space \(T^*(G_1\times G_2)\). Strictly speaking one should subtract the collision set from this space, but we will not concern ourselves with this for now. By identifying tangent spaces with their duals using the inner product in \(\Q\), the phase space may be identified with the set
\begin{equation}
M=\left\{
(g_1,p_1,g_2,p_2)\in\Q^4~|~g_1,g_2\in G,~\langle p_1,g_1\rangle=\langle p_2,g_2\rangle=0
\right\}.
\end{equation}
The position vectors for both particles are \(g_1\) and \(g_2\), and the linear momenta \(p_1\) and \(p_2\) dynamically given by \(m_1\dot{g}_1\) and \(m_2\dot{g}_2\) respectively. The dynamics are determined by the hamiltonian
\begin{equation}\label{ham}
H(g_1,p_1,g_2,p_2)=\frac{|p_1|^2}{2m_1}+\frac{|p_2|^2}{2m_2}+V(g_1,g_2).
\end{equation}
Here \(|p|^2=\langle p,p\rangle\), and \(V(g_1,g_2)\) is a function of the distance \(|g_2-g_1|\).
\subsection{Symmetries and one-parameter subgroups}
Owing to the reformulation of the problem in terms of quaternions, the \(SO(4)\)-symmetry on the configuration space may be realised by the action of its well-known double cover, \(G\times G\). Explicitly this is given by \(\Phi\colon G\times G\rightarrow SO(4)\), where
\begin{equation}
\label{so4action}
\Phi(l,r)\cdot q=lqr^{-1}
\end{equation}
for \(q\in\Q\cong\R^4\). The cotangent lift of this action to \(M\) acts diagonally on each component, and a quick check confirms that it indeed preserves the hamiltonian. We will from now on write the symmetry group as \(G_L\times G_R\) to distinguish it from the configuration space \(G_1\times G_2\). 

The space \(\Imag\Q\) of purely imaginary quaternions is equal to the Lie algebra \(\mathfrak{g}\) of \(G\), and the adjoint action given by \(\Phi(g,g)\). The infinitesimal adjoint action is obtained by differentiating \(g\) to give
\begin{equation}
\label{crossproduct}
\ad_\omega q=\omega q-q\omega=[\omega,q]
\end{equation}
for \(q,\omega\in\Imag\Q\). By identifying \(\Imag\Q\) with \(\R^3\) in the obvious sense, the adjoint action is related to the cross-product by \([\omega,q]=2(\omega\times q)\).

The adjoint action of \(G\) acts transitively on each sphere of imaginary quaternions of a given length. It follows that every one-parameter subgroup of \(G_L\times G_R\) is conjugate to one of the form \(\{(e^{it\eta},e^{it\xi}),t\in\R\}\) for some \(\eta,\xi\ge 0\). With the aid of \eqref{so4action} one sees that the action of this subgroup on \(\Q\) preserves the mutually orthogonal, oriented planes \(\C=\Span\{1,i\}\) and \(\C j=\Span\{j,k\}\). This action rotates \(\C\) and \(\C j\) through an angle of \(\xi-\eta\) and \(\xi+\eta\) respectively with each unit of time.

Given any one-parameter subgroup conjugate to that given above, we categorise it into one of the following four types: \emph{trivial} for when \(\xi=\eta=0\); a \emph{simple rotation} for \(\xi=\eta\ne 0\); an \emph{isoclinic rotation} for when precisely one of either \(\xi\) or \(\eta\) is equal to zero; and finally, the generic subgroup is called a \emph{double rotation} for when \(\xi,\eta\ne 0\) and \(\xi\ne\eta\). 

\subsection{Reduction and relative equilibria}
It is curious that the symmetry group and configuration space are both the same. This is identical to the familiar situation of cotangent bundle reduction of a group under left/right multiplication. However, for our example the essential difference is that the group action is given by \emph{simultaneous} left and right diagonal multiplication. This complicates the picture somewhat; in particular, this group action is not free. The points at which the action is not free may be characterised with the following argument: \(g_1\) and \(g_2\) belong to some plane and thus, the isotropy group fixing these two points includes rotations in the orthogonal plane. The action therefore fails to be free if and only if the momenta \(p_1\) and \(p_2\) have no component in this orthogonal plane, and thus, all vectors are coplanar. We will call such points \emph{cocircular} as the resulting motion remains inside a great circle on the sphere.

Nonetheless, the actions of left and right multiplication, given by restriction to one of the \(G_L\) or \(G_R\) subgroups, is free. For a group acting on its cotangent bundle by left/right cotangent lift, the momentum map is given by right/left translation back to the origin \cite{arnold}. As the \(G_L\)- and \(G_R\)-actions are the product of two copies of left and right multiplication respectively, the left momentum map is given by
\begin{equation}
\label{left_mom}
J_L(g_1,p_1,g_2,p_2)= \underbrace{p_1g_1^{-1}}_{L_1}+\underbrace{p_2g_2^{-1}}_{L_2}=\lambda\in\mathfrak{g}_L^*,
\end{equation}
and the right momentum map by
\begin{equation}
\label{right_mom}
J_R(g_1,p_1,g_2,p_2)= \underbrace{g_1^{-1}p_1}_{R_1}+\underbrace{g_2^{-1}p_2}_{R_2}=\rho\in\mathfrak{g}_R^*.
\end{equation}
We write \(L_i=p_ig_i^{-1}\) and \(R_i=g_i^{-1}p_i\) to denote the \emph{left} and \emph{right momentum} of the \(i\)\textsuperscript{th}-particle respectively. The \emph{total left} and \emph{total right momenta} \(\lambda\) and \(\rho\) are both first integrals. As the actions of left and right multiplication are each free and proper, we may safely define the \emph{left} and \emph{right reduced spaces} \(M_\lambda\) and \(M_\rho\) respectively. 

The left and right reduced spaces both inherit a group action from the residual right and left symmetry. These reduced spaces can therefore be reduced again in stages. From the Commuting Reduction Theorem \cite{bigstages}, the momentum map for the full symmetry group is \(J_{L,R}=J_L\times J_R\), and for when the action is free and proper, the staged reduced spaces \((M_\lambda)_\rho\) and \((M_\rho)_\lambda\) are both symplectomorphic to the `one-shot' \emph{full reduced space} \(M_{\lambda,\rho}\). We would therefore like to understand the set of critical values and points of \(J_{L,R}\).
\begin{propn}\label{crit}
	The set of critical values of \(J_{L,R}\) are those \((\lambda,\rho)\) with \(|\lambda|=|\rho|\). The pre-image of this set consists of all \((g_1,p_1,g_2,p_2)\) belonging to a common 3-dimensional subspace, and thus correspond to solutions contained within an equatorial 2-sphere. For this reason we will refer to such critical points as being \emph{cospherical}.
\end{propn}
\begin{proof}
	The momentum map has critical values on points at which the action is not locally free \cite{critreg}. As we have seen, these are the cocircular points. One may suppose the momenta and positions of such a point are contained to the complex plane in \(\Q\), from which it follows from the definitions that \(\lambda=\rho\). As the momentum map is equivariant, taking the orbit through these values gives us the desired set of critical values.
	
	From the definition of \(\lambda\) and \(\rho\), we may write
	\begin{equation}
	|\lambda|^2-|\rho|^2=2\langle L_1,L_2\rangle-2\langle R_1,R_2\rangle.
	\end{equation}
	We may suppose \(g_1=1\) and \(p_1\) is purely imaginary. By writing the imaginary part of \(g_2\) as \(\overline{g}_2\) the expression above may be rewritten as
	\[
	4\langle p_1,\overline{g}_2\times p_2\rangle.
	\]
	This is equal to zero if and only if \(p_1\), \(p_2\) and \(\overline{g}_2\) span a common plane in \(\Imag\Q\). However, this is equivalent to \(g_1\), \(p_1\), \(g_2\) and \(p_2\) belonging to a common 3-dimensional subspace given by the span of this plane together with the real line. Hence, by equivariance, the set of cospherical points is exactly the set of critical points of the momentum map.
\end{proof}
A relative equilibria (RE) in a symplectic manifold with a hamiltonian group action is a solution which is also the orbit under the action of a one-parameter subgroup of the group of symmetries \cite{marsden_1992}. Equivalently, it is a solution which projects to a point in the reduced space. For our problem, the right and left multiplication naturally descend to give well-defined actions on the left and right reduced spaces respectively. It follows that RE in \(M\) with respect to the \(G_L\times G_R\)-action project into RE in both the left and right reduced spaces. In fact, the converse is also true. 
\begin{propn}\label{REpropn}
	Any RE in \(M\) projects to RE in both the left and right reduced spaces. Conversely, any RE in any of the left or right reduced spaces is the projection of a RE in \(M\).
\end{propn}
\begin{proof}
	The proof follows from commutativity of the diagram in Figure~\ref{commute} which consists of canonical projection maps onto orbit quotients, and the definition of a RE as a fixed point in a reduced space.
\end{proof}
The task of classifying RE in \(M\) is therefore equivalent to that of finding all RE in any one of the left or right reduced spaces. This will turn out to be more tractable than trying to equivalently classify all of the fixed points on the full reduced space.
\begin{figure}
	\begin{center}
		\begin{tikzcd}
		& M \arrow[dd,"\pi_{G_L\times G_R}"]\arrow[rd,"\pi_{G_R}"]\arrow[ld,"\pi_{G_L}"']& \\ M/G_L\arrow[rd,"\pi_{G_R}"'] & & M/G_R \arrow[ld,"\pi_{G_L}"]\\ & M/(G_L\times G_R)&
		\end{tikzcd}
	\end{center}
\caption{\label{commute} }
\end{figure}
\subsection{The Lagrange top}
The double cover in \eqref{so4action} is a local diffeomorphism and so lifts to a double cover of cotangent bundles which is a local symplectomorphism. Furthermore, as the hamiltonian factors through this double cover, the dynamics factor through as well. One may also see from \eqref{so4action} that the left and right \(G\)-symmetry descends through the double cover to give the left and right multiplication of \(SO(3)\) on \(SO(4)\), where \(SO(3)\) is the subgroup fixing the real line in \(\Q\cong\R^4\). We thus have a dynamical system on \(SO(4)\) with a left and right \(SO(3)\)-symmetry. This situation should feel familiar: it is the exact same situation we have for a symmetric spinning top in 4 dimensions, as we now recall.

The configuration of a rigid body with a fixed point at the origin in \(\R^4\) may be determined by an element in \(SO(4)\), that element being the transformation which sends the body from a given initial state to its current one. For when the body is under the influence of a potential which is a function of `height' in \(\R^4\), for which the direction of increasing height we will call the \emph{vertical}, the dynamical system is that of the \emph{heavy top}. This system is invariant under left multiplication of the \(SO(3)\) subgroup which fixes the vertical. If the body is also invariant about rotations through a line within the body, which we shall call \emph{the body axis}, then we have the 4-dimensional generalisation of the \emph{Lagrange top}. We may suppose that the body axis and the vertical are aligned when the body is in its initial identity configuration. In this way, the full symmetry of the system is both left and right multiplication of the \(SO(3)\) subgroup fixing the vertical.

The study of the ordinary Lagrange top in 3 dimensions is old and well understood. We recommend the modern accounts of the problem given in \cite{cushbook,ratiu1982lagrange,heavytopgeomtreatment}. Recently there has been some attention given to the higher dimensional generalisations of the spinning top \cite{somerecent}. The higher dimensional version of the Lagrange top, as we have defined it, was studied by Beljaev in \cite{belyaev} and shown to be integrable. We note that an alternative generalisation is given in \cite{ratiuintegrable} which is also shown to be integrable. 

The aim now is to find the hamiltonian on \(M\) whose dynamics project through the double cover to give the Lagrange top dynamics on \(SO(4)\). To do this, we describe the hamiltonian on \(T^*SO(4)\) and pull it back under the double cover. When the potential is linear in height the hamiltonian is given by
\begin{equation}\label{lagrangian}
\frac{1}{2}\langle L,\mathbb{I}^{-1}(L)\rangle+\gamma\langle ac_0,v_0\rangle.
\end{equation}
Here \(a\in SO(4)\) is the configuration of the body, \(c_0\) and \(v_0\) are the initial centre of mass and vertical vectors, \(\gamma>0\) a constant, and \(\mathbb{I}\colon\mathfrak{so}(4)\rightarrow\mathfrak{so}(4)^*\) is the inertia tensor of the body, where \(L\) is the angular momentum in the body frame. Identifying \(\mathfrak{so}(4)\) with its dual using the standard trace form gives 
\begin{equation}
\label{inertia}
\mathbb{I}(\omega)=\frac{1}{2}(A\omega+\omega A)
\end{equation}
for \(A=\text{diag}(1,1,1,I_4)\). Here \(I_4\) is the moment of inertia of the body aligned along the vertical, and where we have chosen to work in units with the other three moments of inertia set to one. Differentiating \(\Phi\) in \eqref{so4action} at the identify gives the well-known isomorphism \(\mathfrak{g}_1\times\mathfrak{g}_2\rightarrow\mathfrak{so}(4)\), the pullback of which is given by
\begin{equation}
\label{so4isom}
\Phi^*;L=\begin{pmatrix}
\widehat{\Omega} & \eta \\ -\eta
^T& 0
\end{pmatrix}\longmapsto (\Omega+\eta,\Omega-\eta).
\end{equation}
Here we are identifying \(\mathfrak{g}=\Imag\Q\) with vectors in \(\R^3\), and where \(\widehat{\Omega}\) denotes the element of \(\mathfrak{so}(3)\) with \(\widehat{\Omega}v=\Omega\times v\) for all \(v\in\R^3\). In our identification of \(\R^4\) with \(\Q\), we are setting the vertical \(v_0=c_0\) to be the real unit \(1\).  It is a routine exercise to show that the right momenta \((R_1,R_2)\) and the body angular momentum \(L\), both obtained by left translation to the identity, are related through the double cover by \((R_1,R_2)=\Phi^*(L)\). Using this identity along with \(a=\Phi(g_1,g_2)\) allows us to pull back the hamiltonian in \eqref{lagrangian} to obtain
\begin{equation}
\label{lagrangeham}
\frac{1+\alpha}{4}(|p_1|^2+|p_2|^2)+\frac{1-\alpha}{2}\langle R_1,R_2\rangle+\gamma\langle g_1,g_2\rangle
\end{equation}
where \(\alpha=2(1+I_4)^{-1}\). Note that this is in the same form of the 2-body hamiltonian in \eqref{ham} except for the presence of the \(\langle R_1,R_2\rangle\) term. We will see later that, although these hamiltonians are different, on the full reduced space they differ by a Casimir, and thus, give the same flow.
\section{Reduction}
\subsection{The left and right reduced spaces}
Motivated by the connection between the symmetry of the problem with that of the Lagrange top, we will emulate a method used for reducing the Lagrange top by the left or right symmetry by using the Semidirect Product Reduction by Stages Theorem as demonstrated in \cite{SDbetter}. In standard treatments of the Lagrange top this theorem identifies the reduced spaces with coadjoint orbits of the special Euclidean group.

\begin{thm}[Semidirect Product Reduction by Stages, \cite{SDbetter}]
	Let \(V\) be a representation of \(H\) and consider the semidirect product \(S=H\ltimes V\). For a given \(p\in V^*\) let \(H_p\) denote the stabiliser of this element with respect to the contragredient representation, and consider the action of \(H_p\) on \(T^*H\) by cotangent lift on the left/right. There is a Poisson immersion of \(T^*H/H_p\) into \(\mathfrak{s}^*_\pm=\mathfrak{h}^*\times V^*\) given by sending the orbit through \(\eta\in T^*_aH\) to 
	\begin{equation}\label{slide}
	\left(\mathcal{L}_{a^{-1}}^*\eta,a^{-1}p\right)\in\mathfrak{s}^*_+\quad\text{or}\quad \left(\mathcal{R}_{a^{-1}}^*\eta,ap\right)\in\mathfrak{s}^*_-
	\end{equation}
	for the left/right case respectively. Here the \(\pm\) sign indicates that the Poisson structure differs by a sign between the spaces, and the \(\mathcal{L}\) and \(\mathcal{R}\) denote the left and right cotangent lifts on \(T^*H\). Moreover, if we let \(\orb\) denote a coadjoint orbit through \(\mu\in\mathfrak{h}_p^*\) and \(J\colon T^*H\rightarrow\mathfrak{h}_p^*\) the momentum map for the action of \(H_p\) on \(T^*H\), then the immersion restricted to \(J^{-1}(\orb)/(H_p)\) establishes a symplectomorphism between this symplectic orbit-reduced space and a coadjoint orbit in \(\mathfrak{s}^*\).
\end{thm}

We apply this theorem directly to the task of reducing \(T^*(G_1\times G_2)\) by the diagonal subgroup \(G\) acting on either the left or right. To do this, we set \(H\) in the theorem to \(G_1\times G_2\) with the representation in \eqref{so4action} on \(V=\Q\). This semidirect product \(S\) is the simply-connected double cover over the special Euclidean group \(SE(4)\). Thanks to the inner product on \(\Q\), we are free to identify spaces with their duals, and verify that the isotropy subgroup of \(1\in\Q^*\) is indeed the diagonal subgroup \(G\). Implementing \eqref{slide} demonstrates that the left and right Poisson reduced spaces are given by the elements 
\begin{equation}\label{leftright}
(R_1,R_2,g_L)\quad\text{and}\quad(L_1,L_2,g_R)
\end{equation}
inside \(\mathfrak{s}^*_+\) and \(\mathfrak{s}^*_-\) respectively. Here we have introduced the respective left- and right-invariant quantities \(g_L=g_1^{-1}g_2\) and \(g_R=g_1g_2^{-1}\). The left reduced space \(M_\lambda\) written as an orbit-reduced space is \(J_L^{-1}(\orb)/G_L\), where \(\orb\) is the coadjoint orbit in \(\mathfrak{g}_L^*\) through \(\lambda\). Therefore, this reduced space is equal to the set of \((R_1,R_2,g_L)\) in \(\mathfrak{s}_+\) with \(|L_1+L_2|^2=|\lambda|^2\). This may be rewritten explicitly in terms of the left-invariant variables as
\begin{equation}\label{lamb}
|L_1+L_2|^2=|p_1g_1^{-1}+p_2g_2^{-1}|^2=|(g_1^{-1}{p}_1)(g_1^{-1}g_2)+(g_1^{-1}g_2)(g_2^{-1}{p}_2)|^2=|R_1g_L+g_LR_2|^2.
\end{equation}
The symplectic reduced space \(M_\lambda\) is thus the subset of \((R_1,R_2,g_L)\) in \(\mathfrak{s}^*_+\) with \(|g_L|^2=1\) and \(|R_1g_L+g_LR_2|^2=|\lambda|^2\). These two functions are the only Casimirs of \(\mathfrak{s}^*_+\). The geometry of the orbits is made clearer by applying the following transformation on \(\mathfrak{s}^*_+\) for \(g_L\ne0\)
\begin{equation}\label{transparent}
(R_1,R_2,g_L)\longmapsto\left((R_1g_L+g_LR_2)g_L^{-1},-R_1+g_LR_2g_L^{-1},g_L\right).
\end{equation}
One may now see that the reduced space is diffeomorphic to \(\orb\times\mathfrak{g}^*\times G\). An entirely similar argument can be made for the right reduced space as that above.
\begin{propn}\label{firstspaces}
	Let typical elements in \(\mathfrak{s}^*_\pm=\mathfrak{g}_1^*\times\mathfrak{g}_2^*\times\Q^*\) be denoted by \((A_1,A_2,g_D)\). There are two Casimir functions given by
	\begin{align*}
	C_1&=|g_D|^2, \\
	C_2&=|A_1g_D+g_DA_2|^2.
	\end{align*}
	The left and right reduced spaces \(M_\lambda\) and \(M_\rho\) are symplectomorphic to the coadjoint orbits in \(\mathfrak{s}^*_\pm\) given by setting \(C_1\) to \(1\), and \(C_2\) to \(|\lambda|^2\) and \(|\rho|^2\) respectively. In each case, the typical elements of the orbit may be identified with the left/right-invariant dynamic variables by setting the ambidextrous dummy variables \((A,D)\) to either \((R,L)\) or \((L,R)\) for the left and right reduced spaces respectively. The coadjoint orbits in question are generically diffeomorphic to \(S^2\times\R^3\times S^3\) for \(C_2\ne 0\), and to \(\R^3\times S^3\) for \(C_2=0\).
\end{propn}
In both the left and right reduced spaces the hamiltonian in \eqref{ham} descends through the reduction procedure to give the reduced hamiltonian on \(\mathfrak{s}^*_\pm\)
\begin{equation}
\label{reduced_ham}
H(A_1,A_2,g_D)=\frac{|A_1|^2}{2m_1}+\frac{|A_2|^2}{2m_2}+V(g_D).
\end{equation}
The function \({V}(g_D)\) is the \emph{reduced potential} defined by \({V}(g_D)=V(1,g_D)\), which for both the left and right-reduced spaces is equal to \(V(1,g_L)=V(g_R,1)=V(g_1,g_2)\) using the left- and right-invariance of \(V\). 

We wish to highlight an interesting feature for when one of the left or right momenta is zero. The corresponding reduced space \(\R^3\times S^3\) is in fact symplectomorphic to \(T^*S^3\) with the canonical symplectic form \cite{bigstages}[Chapter 4]. Furthermore, for when the masses are equal, the reduced hamiltonian on \(T^*S^3\) gives the same dynamical system as the Kepler one-body problem with the second body fixed at \(1\in S^3\). This should be contrasted with what was said at the beginning: in Euclidean space the 2-body problem may be reduced to the Kepler problem by transforming into a centre of mass frame. Here we have a kind of analogue to this, that when one of the left or right momenta is zero, the corresponding reduced space gives the standard Kepler problem on the sphere.

\subsection{The full reduced space}
Now we consider the task of reducing the left/right reduced space by the residual right/left symmetry. Without any loss of generality, we focus on reducing the left Poisson reduced space \(M/G_L\) by the group \(G_R\) of right translations. From the definitions of the left-invariant variables this group action descends to \(\mathfrak{s}^*_+\) as
\begin{equation}
\label{gr_action}
r\cdot(R_1,R_2,g_L)=(rR_1r^{-1},rR_2r^{-1},rg_Lr^{-1})
\end{equation}
for \(r\in G_R\). From the Commuting Reduction Theorem the momentum map for this action is also given by the total right momentum \(\rho=R_1+R_2\).

We must now confront the issue that this action is not free. By writing an element as \((R_1,R_2,\overline{g}_L,r)\), where we have decomposed \(g_L\) into its imaginary and real parts, the \(G_R\)-action on \(\mathfrak{s}^*\) decomposes into the irreducible pieces \(\mathfrak{g}_1\times\mathfrak{g}_2\times\mathfrak{g}_3\times\R\); that is, three copies of the adjoint representation of \(G\), and the trivial representation. The adjoint representation of \(G\) factors through the double cover \(G\rightarrow SO(3)\) to give the standard vector representation of \(SO(3)\) on \(\mathfrak{g}=\Imag\Q\cong\R^3\). It follows that this action fails to be free whenever \(R_1\), \(R_2\) and \(\overline{g}_L\) are colinear. In order to handle these singular cases we will employ the method of universal reduction from \cite{universalred} via the use of invariant theory to obtain the orbit quotient as a semialgebraic variety. This technique is demonstrated in, for example \cite{examples}, and similarly in \cite{cushbook} for the second stage of reduction for the ordinary Lagrange top. 

Temporarily denote elements in \(\mathfrak{g}_1\times\mathfrak{g}_2\times\mathfrak{g}_3\) by \((v_1,v_2,v_3)\). The First Fundamental Theorem of Invariant Theory for the special orthogonal group \cite{kraft} tells us that the invariant ring is generated by the pairwise inner products \(k_{ij}=\langle v_i,v_j\rangle\) for \(i\le j\), and the determinant \(\delta=\langle v_1\times v_2,v_3\rangle\). These are not independent as they satisfy the algebraic relation \(\delta^2=\text{det}(k_{ij})\). As the group \(G_R\) is compact, the orbits are separated by the values taken by the generators of the invariant ring, and hence, the quotient \(\mathfrak{s}^*_+/G_R\) may be identified with the image of the Hilbert map
\begin{equation*}
\sigma\colon \mathfrak{s}^*_+\longrightarrow\R^8;\quad(v_1,v_2,v_3,r)\longmapsto\left(\{k_{ij}\}_{i\le j},\delta,r\right).
\end{equation*}
The image of this map is the semialgebraic variety defined by those points satisfying \(\delta^2=\text{det}(k_{ij})\), and the inequalities \(k_{ij}^2\le k_{ii}k_{jj}\) for each \((i,j)\)-pair.

Since the momentum \(\rho=R_1+R_2\) is conserved, the \(G_R\)-invariant quantity \(|\rho|^2\) descends to \(\mathfrak{s}^*_+/G_R\) to give an additional Casimir
\[
C_3=|R_1+R_2|^2=k_{11}+2k_{12}+k_{22}.
\]
The two Casimirs \(C_1\) and \(C_2\) in Proposition~\ref{firstspaces} also descend to the full reduced space. The first of these is easily seen to be 
\[
C_1=|g_L|^2=k_{33}+r^2.
\]
The second however, requires a special effort to express in terms of the invariant generators.
\begin{lem}
	The Casimir \(C_2=|R_1g_L+g_LR_2|^2\) may be expressed in terms of the generators of the \(G_R\)-invariant ring on \(\mathfrak{s}^*_+\) as
	\begin{equation*}
	C_2=(k_{33}+r^2)(k_{11}+k_{22})+2k_{12}(r^2-k_{33})+4k_{13}k_{23}-4r\delta.
	\end{equation*}
\end{lem}
\begin{proof}
	Expanding \(C_2\) gives
	\[
	|R_1g_L|^2+|g_LR_2|^2+2\langle R_1g_L,g_LR_2\rangle=(k_{33}+r^2)(k_{11}+k_{22})+2\langle R_1g_L,g_LR_2\rangle.
	\]
	Rewriting \(g_L\) as \(r+\overline{g}_L\), the final term above may be written as
	\[
	2r^2\langle R_1,R_2\rangle+2r\left(\langle R_1\overline{g}_L,R_2\rangle+\langle R_1,\overline{g}_LR_2\rangle\right)+2\langle R_1\overline{g}_L,\overline{g}_LR_2\rangle=2r^2k_{12}-4r\delta+2\langle R_1\overline{g}_L,\overline{g}_LR_2\rangle.
	\]
	By multiplying out the cross-product terms, it is possible to establish the following identity
		\[
	2\langle R_1\overline{g}_L,\overline{g}_LR_2\rangle=4\langle R_1\times\overline{g}_L,\overline{g}_L\times R_2\rangle+2\langle\overline{g}_L,\overline{g}_L\rangle\langle R_1,R_2\rangle.
	\]
	Finally, using the vector quadruple product \(\langle a\times b,c\times d\rangle=\langle a,c\rangle\langle b,d\rangle-\langle a,d\rangle\langle b,c\rangle\) in the expression above gives the desired result.
\end{proof}
\begin{thm}\label{full_red}
	The Poisson reduced space \(\mathfrak{s}^*_+/G_R\) is the semialgebraic variety given by coordinates \((\{k_{ij}\}_{i\le j},\delta,r)\) in \(\R^8\) satisfying \(\delta^2=\det(k_{ij})\) and \(k_{ij}^2\le k_{ii}k_{jj}\). There are three Casimirs,
	\begin{align*}
	C_1&=k_{33}+r^2, \\
	C_2&=(k_{33}+r^2)(k_{11}+k_{22})+2k_{12}(r^2-k_{33})+4k_{13}k_{23}-4r\delta,\\
	C_3&=k_{11}+k_{22}+2k_{12}.
	\end{align*}
	The full reduced space \((M_\lambda)_\rho\) is obtained by setting \(C_1=1\), \(C_2=|\lambda|^2\), and \(C_3=|\rho|^2\).
\end{thm}
For \(|\lambda|,|\rho|\ne 0\), these typical reduced spaces are 4-dimensional. The algebraic awkwardness of the Casimir \(C_2\) together with the relation \(\delta^2=\text{det}(k_{ij})\) makes it difficult to grasp the geometry of these reduced spaces. Indeed, it would be of considerable interest to be able to say more about them. Nonetheless, below we describe the degenerate 2-dimensional reduced spaces.

Consider the full reduced space \((M_\lambda)_\rho\) for \(\rho=0\). By applying the algebraic inequalities in Theorem~\ref{full_red} for \(C_3=0\) we obtain an additional two constraints: \(k_{11}=k_{22}\) and \(k_{13}=-k_{23}\). After eliminating variables the reduced space is found to be homeomorphic to the set of points \((k_{11},k_{13},\theta)\) satisfying
\begin{equation}
\label{kepler_spaces}
4k_{11}=\frac{|\lambda|^2+4k_{13}^2}{\sin^2\theta},
\end{equation}
where we write \(k_{33}=\sin^2\theta\). When \(|\lambda|\) is non-zero this leaf is homeomorphic to \(\R^2\). On the other hand, if \(\lambda=0\) this leaf degenerates into the singular canoe shown in Figure~\ref{bowl}. As we have previously remarked, the reduced spaces \(M_0\) are symplectomorphic to \(T^*S^3\), and for when the masses are equal the resulting dynamical system is the Kepler problem on the sphere. These reduced spaces we have described therefore coincide with those for the Kepler problem.

\begin{figure}
	\centering
	\begin{subfigure}{0.4\textwidth}
		\includegraphics[width=\textwidth]{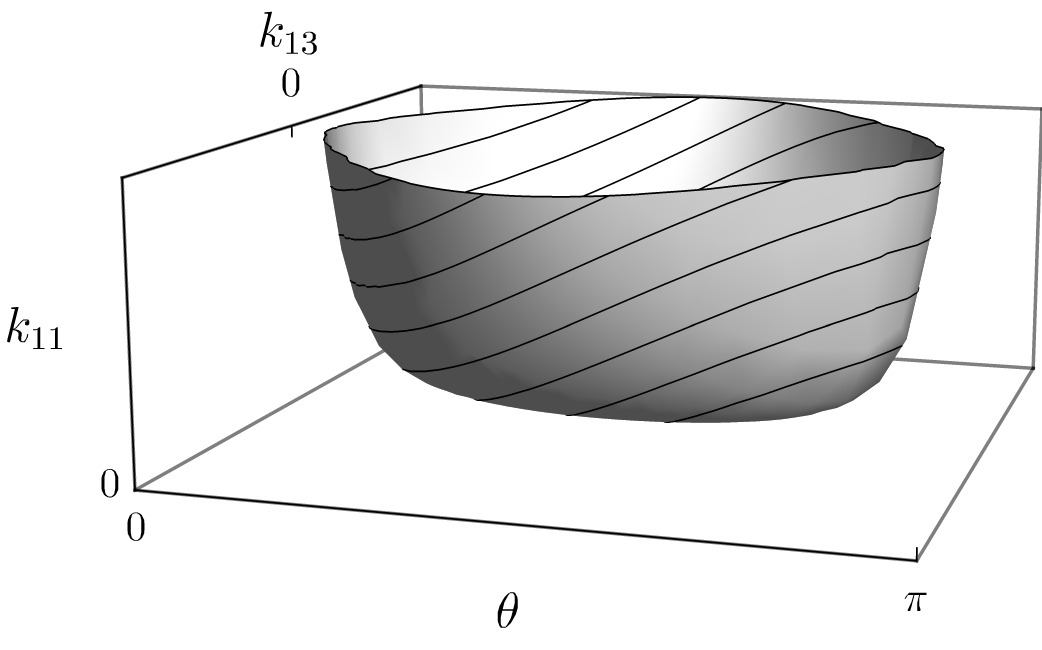}
	\end{subfigure}
	\qquad\quad
	\begin{subfigure}{0.4\textwidth}
		\includegraphics[width=\textwidth]{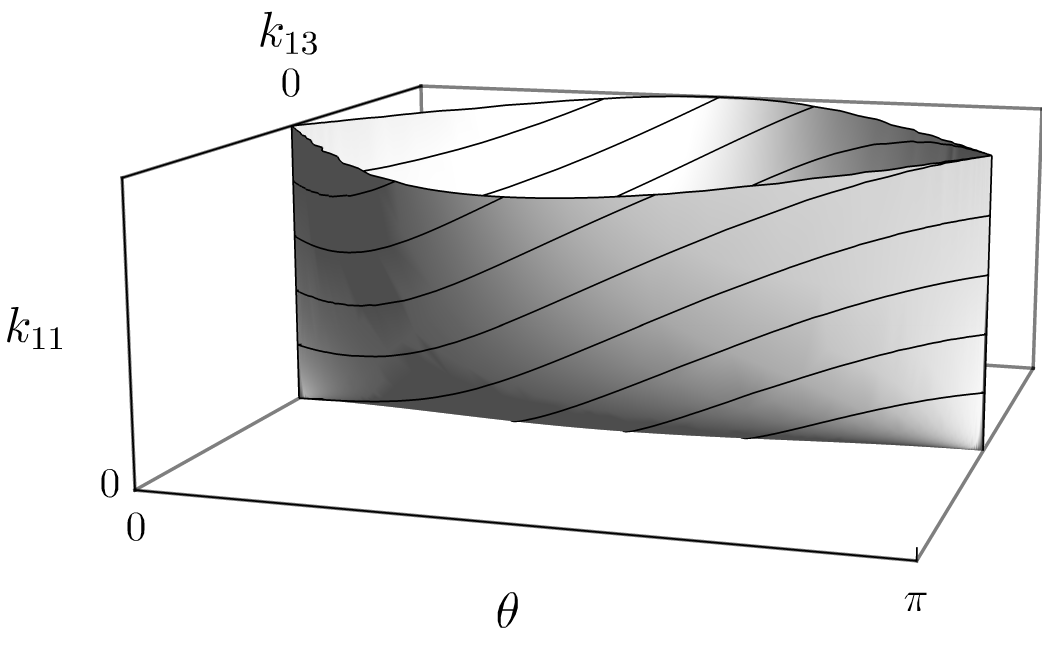}
	\end{subfigure}
	\caption{\label{bowl} The full reduced spaces \((M_\lambda)_\rho\) when \(\rho=0\). These spaces are 2-dimensional and homeomorphic to the plane for \(\lambda\ne 0\) as shown in the figure on the left. The leaf degenerates at \(\lambda=0\) into the canoe as shown on the right. The contours are the level sets of the altered hamiltonian in \eqref{alt_ham} for the Lagrange top.}
\end{figure}

\subsection{The singular strata}

An advantage of using invariants to describe the reduced space as a semialgebraic variety, is that it includes those points at which the action is not free. Consequently, these reduced spaces are not smooth in general, but stratified symplectic spaces. The theory of such stratified spaces is detailed in \cite{sjamaar}. It is shown that the strata of a reduced space, which are invariant under the dynamics, correspond to the different possible isotropy subgroups of the action. We now discuss each of the possible isotropy subgroups for the \(G_R\)-action in \eqref{gr_action} and their corresponding strata in turn.

From Proposition~\ref{crit}, the action is free whenever \(|\lambda|\ne|\rho|\) and not free precisely on those points which are cocircular. It follows that the reduced spaces \(M_{\lambda,\rho}\) for \(|\lambda|\ne|\rho|\) consist of a single stratum and are thus bonafide smooth manifolds. For when \(|\lambda|=|\rho|\), the reduced space contains an open dense stratum corresponding to the non-cocircular points at which the action is free.

The stratum corresponding to when the isotropy subgroup is all of \(G_R\) is for when \(R_1=R_2=\overline{g}_L=0\), and thus consists solely of the two points \(r=\pm 1\), where the other generators are zero. These two points are the corners of the canoe in \(M_{0,0}\) in Figure~\ref{bowl}, and correspond to the states where the two particles are motionless and either antipodal, or in the same position.

The strata corresponding to when the isotropy subgroup is \(SO(2)\) are for those points where \((R_1,R_2,\overline{g}_L)\) are colinear and not all zero. It follows that the generators of the invariant ring satisfy three further relations, given by changing the inequalities in Theorem~\ref{full_red} into equalities. In fact, as \(\delta=0\), and since \(|\lambda|=|\rho|\) at this point, the relation \(k_{11}k_{33}=k_{13}^2\) is not independent and so we have 6 constraints in total. Indeed, after eliminating variables one may show that there are two degrees of freedom given by \(k_{11}+k_{22}\) and \(\theta\), where we are writing \(r=\cos\theta\). For when \(|\lambda|=|\rho|\ne 0\) this stratum is homeomorphic to a cylinder as shown in Figure~\ref{tube}, and degenerates into the canoe, minus the corners, when the momentum is zero. These strata correspond to the cocircular configurations of the particles.
\begin{figure}
	\centering
	\begin{subfigure}{0.3\textwidth}
		\includegraphics[width=0.7\textwidth]{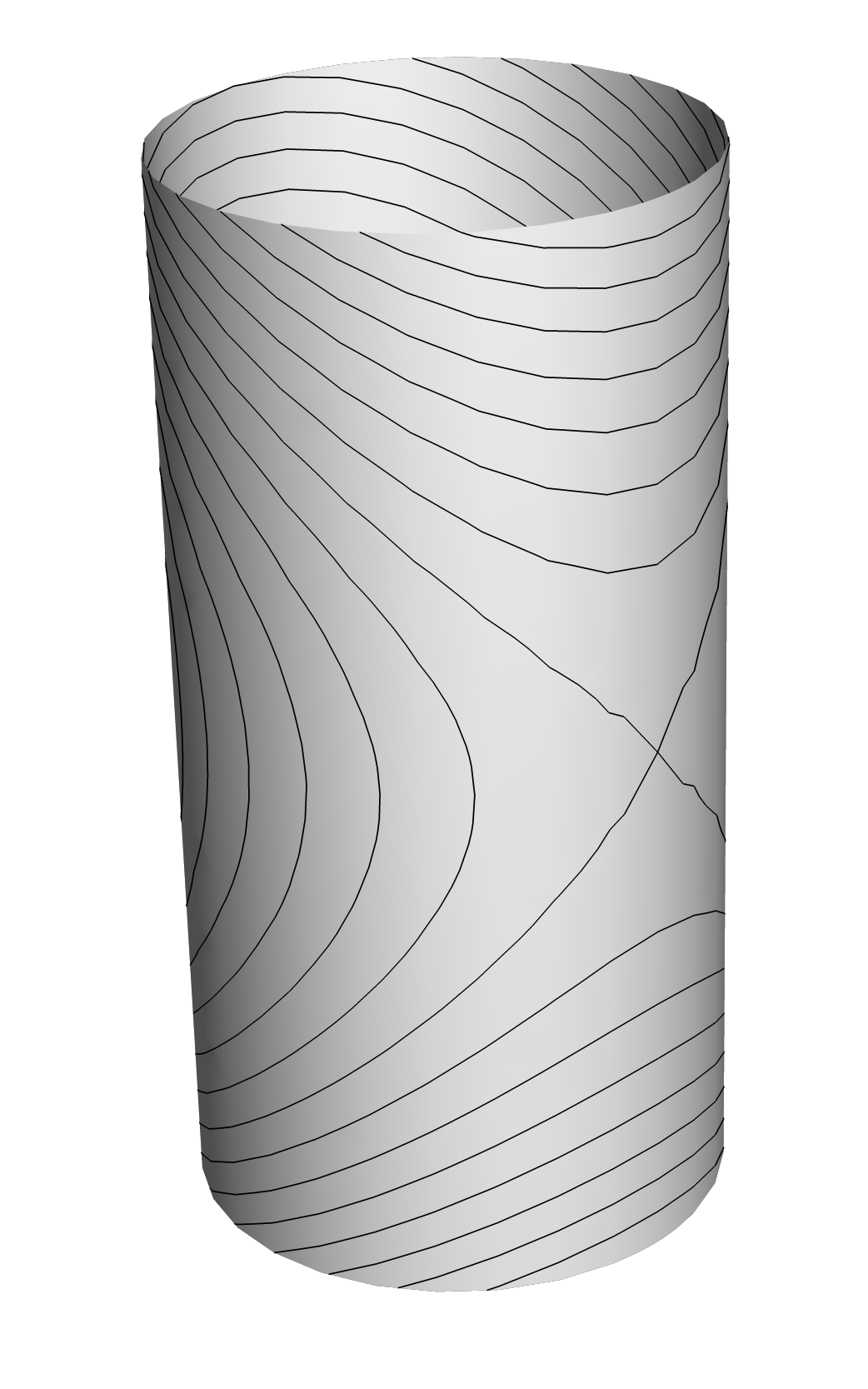}
	\end{subfigure}
	~
	\begin{subfigure}{0.3\textwidth}
		\includegraphics[width=\textwidth]{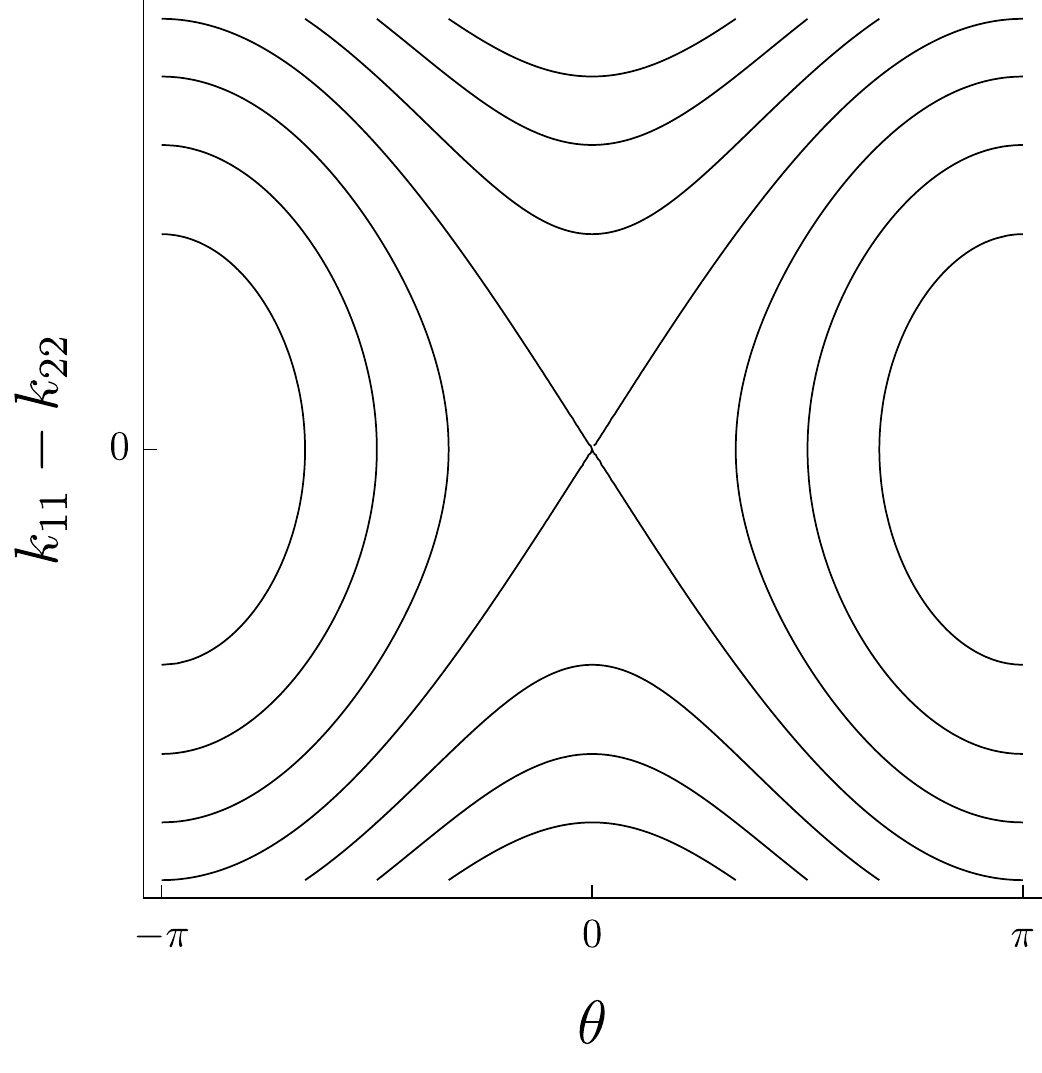}	
	\end{subfigure}
	\caption{\label{tube}The singular stratum inside \(M_{\lambda,\rho}\) for \(|\lambda|=|\rho|\ne 0\) is homeomorphic to a cylinder. In this figure we have added contours for the level sets of the altered hamiltonian in \eqref{alt_ham} for the Lagrange top. }
\end{figure}

\subsection{The equations of motion and the Poisson structure}
 We will now derive hamilton's equations, first on the left and right reduced spaces, and then on the full reduced space. To begin with, we need the Poisson structure between two functions \(f\) and \(g\) on \(\mathfrak{s}^*_\pm\). By the definition of the Lie-Poisson structure, \(\{f,g\}\) evaluated at \((A_1,A_2,g_D)\) in \(\mathfrak{s}^*_\pm\) is given in \cite{semidirect} by
\begin{equation}\label{sbracket}
\pm\langle A_1,[\nabla_1f,\nabla_1g]\rangle\pm\langle A_2,[\nabla_2f,\nabla_2g]\rangle\pm\langle g_D,(\nabla_1 f\nabla_3g-\nabla_3g\nabla_2f)-(\nabla_1 g\nabla_3f-\nabla_3f\nabla_2g)\rangle.
\end{equation}
Here we have written \(\nabla_i\) for \(i=1,2,3\) to mean the gradient of a function on \(\mathfrak{s}^*=\mathfrak{g}_1^*\times\mathfrak{g}_2^*\times\Q^*\) with respect to the three component spaces. The equations of motion may be derived from the Poisson bracket, where the convention we use is \(\dot{f}=\{H,f\}\). Before doing this, we pause to consider the term \(\nabla_3{V}\) from the hamiltonian in \eqref{reduced_ham}. Since \({V}=V(1,g_D)\) is a function of the distance from \(1\) to \(g_D\), it is only a function of the real part \(r\) of \(g_D\). The gradient \(\nabla_3{V}\) is therefore the purely real quaternion \(d{V}/dr\). Armed with this foresight, we may proceed to write down hamilton's equations on each of \(\mathfrak{s}^*_+\) and \(\mathfrak{s}^*_-\).
\begin{equation}\label{Ham_eqns}
\arraycolsep=15pt\def\arraystretch{2}
\begin{array}{ll}
\dot{R}_1=+f(r)\overline{g}_L & \dot{L}_1=-f(r)\overline{g}_R\\
\dot{R}_2=-f(r)\overline{g}_L& \dot{L}_2=+f(r)\overline{g}_R\\
\displaystyle\dot{g}_L=-\frac{R_1}{m_1}g_L+  g_L\frac{R_2}{m_2} &
\displaystyle\dot{g}_R=+\frac{L_1}{m_1}g_R- g_R\frac{L_2}{m_2}
\end{array}
\end{equation}
Here we are writing \(\overline{g}_D\) to mean the imaginary part of \(g_D\), and suggestively abbreviating \(-d{V}/dr\) with \(f(r)\) for \emph{force}. Using these equations together with the definitions of the generators of the invariant ring, we can write the full reduced equations of motion in \(\mathfrak{s}^*/G_R\), which are given below.
\begin{equation}\label{fulll_red_eqns}
\def\arraystretch{2.4}
\begin{array}{ll}
\displaystyle\dot{k}_{11}=2fk_{13}&\displaystyle\dot{r}=\frac{k_{13}}{m_1}-\frac{k_{23}}{m_2} \\\displaystyle
\dot{k}_{12}=f(k_{23}-k_{13})&\displaystyle\dot{\delta}=
\left(\frac{k_{12}k_{13}-k_{11}k_{23}}{m_1}\right)+\left(\frac{k_{13}k_{22}-k_{12}k_{23}}{m_2}\right) \\\displaystyle
\dot{k}_{13}=fk_{33}-r\left(\frac{k_{11}}{m_1}-\frac{k_{12}}{m_2}\right)-\frac{\delta}{m_2} &\displaystyle\\\displaystyle
\dot{k}_{22}=-2fk_{23}&\displaystyle\\\displaystyle
\dot{k}_{23}=-fk_{33}-r\left(\frac{k_{12}}{m_1}-\frac{k_{22}}{m_2}\right)+\frac{\delta}{m_1} &\displaystyle\\\displaystyle
\dot{k}_{33}=2r\left(\frac{k_{23}}{m_2}-\frac{k_{13}}{m_1}\right)&
\end{array}
\end{equation}
These equations give the flow generated by the hamiltonian in \eqref{reduced_ham}, which descends to \(\mathfrak{s}^*_+/G_R\) as
\begin{equation}
\label{full_red_ham}
\frac{k_{11}}{2m_1}+\frac{k_{22}}{2m_2}+V(r).
\end{equation}
It is important to appreciate that although these reduced spaces are considerably smaller than the original phase space, the dimension dropping from 12 to 4, there is a trade-off: the Poisson bracket on \(\mathfrak{s}^*_+/G_R\) between the generators is extremely cumbersome. By definition, this Poisson bracket descends from the bracket on \(\mathfrak{s}^*_+\) as given in \eqref{sbracket}. The difficulty lies in the fact that the invariants are quadratic in \(\mathfrak{s}^*_+\) and thus the bracket between them is cubic. This makes deriving a general formula for the Poisson bracket on \(\mathfrak{s}^*_+/G_R\) a rather Herculean task, and one that this author has failed to complete. As consolation, we offer instead the structure matrix in Table~\ref{table} listing the Poisson bracket between the generators of the invariant ring. It should be emphasised that this Poisson bracket should only be expected to satisfy the Jacobi identity and form a Lie algebra when the generators satisfy the algebraic relations in Theorem~\ref{full_red}.
\begin{table}
	\begin{align*}
	&
	\arraycolsep=5pt\def\arraystretch{1.4}
	\begin{array}{c|cccccc}
	\{~,~\} & k_{11} & k_{12} & k_{13} & k_{22} & k_{23} & k_{33} \\  \hline
	k_{11} & \cdot & 0 & -2rk_{11} & 0 & 2\delta-2rk_{12} & -4rk_{13} \\
	k_{12} & \cdot & \cdot & r(k_{11}-k_{12}) & 0 & r(k_{12}-k_{22}) & 2r(k_{13}-k_{23})\\
	k_{13} & \cdot & \cdot & \cdot & 2\delta-2rk_{12} & -r(k_{13}+k_{23}) & -2rk_{33}\\
	k_{22} & \cdot & \cdot & \cdot & \cdot & 2rk_{22} & 4rk_{23}\\
	k_{23} & \cdot & \cdot & \cdot & \cdot & \cdot & 2rk_{33} \\
	k_{23} & \cdot & \cdot & \cdot & \cdot & \cdot & \cdot \\ 
	\end{array}
	\\ \\&
	\arraycolsep=5pt\def\arraystretch{1.4}
	\begin{array}{c|ccc}
	\{k_{ij},r\} & 1 & 2 & 3 \\ \hline
	1 & 2k_{13} & k_{23}-k_{13} & k_{33} \\
	2 & \cdot & -2k_{23} & -k_{33}\\
	3 & \cdot & \cdot & 0 \\
	\end{array}
	\\ \\&
	\arraycolsep=5pt\def\arraystretch{1.4}
	\begin{array}{c|ccc}
	\{k_{ij},\delta\} & 1 & 2 & 3 \\ \hline
	1 & 2(k_{12}k_{13}-k_{11}k_{23}) & (k_{11}+k_{12})k_{23}-(k_{12}+k_{22})k_{13} &(k_{11}+k_{12})k_{33}-(k_{13}+k_{23})k_{13} \\
	2 & \cdot & 2(k_{13}k_{22}-k_{12}k_{23}) & (k_{13}+k_{23})k_{23}-(k_{12}+k_{22})k_{33}\\
	3 & \cdot & \cdot & 0 \\ 
	\end{array}
	\\ \\&
	\arraycolsep=5pt\def\arraystretch{1.4}
	\begin{array}{c|c}
	\{~,~\} & \delta \\ \hline
	r & 0 
	\end{array}
	\end{align*}
	\caption{\label{table} The Poisson bracket between generators of the \(G_R\)-invariant ring on \(\mathfrak{s}^*_+\) and thus the Poisson structure on the full reduced space \(\mathfrak{s}^*_+/G_R\).}
\end{table}
\subsection{A reprise of the Lagrange top}
The hamiltonian given in \eqref{lagrangeham} for the Lagrange top, descends through this whole reduction procedure to give
\begin{equation}\label{Lagrange_ham}
H=\frac{1+\alpha}{4}(k_{11}+k_{22})+\frac{1-\alpha}{2}k_{12}+\gamma r.
\end{equation}
This is of course, as we have already remarked, different to the 2-body hamiltonian in \eqref{full_red_ham}. However, the two hamiltonians differ by a constant multiple of \(C_3\). Since a Casimir trivially generates a stationary flow, the flow on the full reduced space for the Lagrange top is equivalently the flow generated by the altered hamiltonian
\begin{equation}\label{alt_ham}
\widetilde{H}=H+\frac{(\alpha-1)}{4}C_3=\frac{\alpha}{2}(k_{11}+k_{22})+\gamma r.
\end{equation}
Rather remarkably, one sees that, although the flows on the left and right reduced spaces for the Lagrange top and the 2-body problem are different, they are the same on the full reduced space. In particular, the fully reduced Lagrange top dynamics are precisely the same as that for the 2-body problem, where \(m_1=m_2=\alpha^{-1}\), and \(V(r)=\gamma r\). We can therefore continue in generality to consider a hamiltonian of the form in \eqref{full_red_ham}, and in doing so, will simultaneously treat both the 2-body problem and the Lagrange top.

For the Lagrange top, \(r=\cos\theta\) where \(\theta\) is the angle which the body axis makes with the vertical. This observation combined with the singular stratum in Figure~\ref{tube} allows us to say some interesting things about the dynamics. Firstly, let us explain the meaning of these singular strata for the Lagrange top. For the 2-body problem the \(G_R\)-isotropy type strata given by the points where \(r=\pm1\) corresponds to when the two stationary particles are either antipodal or in the same position. For the Lagrange top these correspond to when the body is statically held  vertically upwards or downwards. 

The \(SO(2)\)-isotropy type strata, which for the 2-body problem corresponds to cocircular arrangements of the two particles, corresponds in the Lagrange-top case to when the motion of the body axis \(g_L\) is longitudinal; that is, when its motion is contained in a plane resulting in a pendulum-like swing of the body. The level sets of \(\widetilde{H}\) on this stratum are given in Figure~\ref{tube}. We can see that the motion is stable when the body is hanging vertically downwards with \(\theta=\pi\). Conversely, when \(\theta=0\) with the body vertically upright, the motion is always unstable. There is therefore, unlike in 3 dimensions, no equivalent notion of the stable `sleeping' motion of the Lagrange top. This result is confirmed later in our stability analysis of the relative equilibria.

Finally we highlight one further remarkable feature of the Lagrange top: that it is completely integrable. Indeed, in \cite{belyaev} it is shown that the \(n\)-dimensional generalisation of the Lagrange top admits a complete set of independent integrals which commute with the left and right \(SO(n-1)\)-symmetry. For the case \(n=4\), this single extra integral is defined on the left reduced space \(\mathfrak{s}^*_+\) by
\begin{equation}
\eta^T\widehat{\Omega}^2\eta+\left(\frac{2\gamma}{\alpha}\right)\overline{g}_L^T\widehat{\Omega}\eta.
\end{equation}
Here we are using the isomorphism in \eqref{so4action} to identify \(\mathfrak{s}^*_+\) with \(\mathfrak{se}(4)^*\) and using the notation for the body angular momentum in \eqref{so4isom}. After a scaling, this may be expressed in terms of the invariant generators as
\begin{equation}
\label{INTEGRAL}
I=\alpha(k_{12}^2-k_{11}k_{22})-2\gamma\delta.
\end{equation}
One may verify directly from the reduced equations in \eqref{fulll_red_eqns} that this is constant for the flow generated by the altered hamiltonian in \eqref{alt_ham}.
\section{Relative equilibria}
\subsection{A classification of the relative equilibria on the left reduced space}
To classify the RE we could of course find the fixed points of the system of full reduced equations in \eqref{fulll_red_eqns}. However, this is easier said than done, and following Proposition~\ref{REpropn} we may equivalently classify the solutions in, say the left reduced space \(M/G_L\) which are the orbits under a one-parameter subgroup of \(G_R\). From the \(G_R\)-action in \eqref{gr_action} we see that such an orbit in \(\mathfrak{g}^*\) is of the form \(q(t)=e^{it\eta}q(0)e^{-it\eta}\) for the generator \(\eta\in\mathfrak{g}_R\). By differentiating this to find the velocity vector, and comparing with the left reduced equations in \eqref{Ham_eqns}, one sees that a solution through \((R_1,R_2,g_L)\) in \(\mathfrak{s}^*_+\) is the orbit of a one-parameter subgroup \(e^{it\eta}\) of \(G_R\) if and only if the following holds
\begin{align}
2(\eta\times R_1)&=+f(r)\overline{g}_L,\label{one}\\
2(\eta\times R_2)&=-f(r)\overline{g}_L,\label{two}\\
2(\eta\times\overline{g}_L)&=-R_1g_L/m_1+g_LR_2/m_2. \label{three}
\end{align}
We suppose without any loss of generality that \(\eta\) is of the form \(|\eta|j\), recalling that \(j\) is an imaginary quaternion in \(\Imag\Q=\mathfrak{g}\). The real part of \(g_L\) is fixed by the action of \(G_R\). We write this real part as \(r=\cos\theta\) for \(\theta\in[0,\pi]\) and classify the solutions of the form above according to the angle \(\theta\). Furthermore, from now on we will suppose that the potential is such that the force \(f(r)\) is never zero.
\subsubsection*{Case 1: $\theta=0,\pi$}
In this case, \(g_L\) is equal to \(\pm1\), and consequently \(\overline{g}_L=0\). It follows from the equations above that \(|\eta|\ge 0\) is arbitrary, and that \(R_1\) and \(R_2\) are scalar multiples of \(j\) which satisfy \(m_2R_1=m_1R_2\).
\subsubsection*{Case 2: $0<\theta<\pi,~\theta\ne\pi/2$}
Equations \eqref{one} and \eqref{two} imply that \(R_1\), \(R_2\) and \(\eta\) are orthogonal to \(\overline{g}_L\). We may also suppose without any loss of generality that \(\overline{g}_L=i\sin\theta\), and thus, that \(g_L=e^{i\theta}\). It then follows that \(R_1\) and \(R_2\) must be of the form
\begin{equation}\label{R1R2form}
R_1=x_1j+yk\quad\text{ and }\quad R_2=x_2j-yk
\end{equation}
for \(x_1\) and \(x_2\) to be determined, and 
\begin{equation}
\label{y}
y=\frac{f\sin\theta}{2|\eta|}.
\end{equation}
It now remains to solve Equation~\eqref{three}, which can be expanded using quaternionic multiplication to give
\begin{align*}
-2|\eta|\sin\theta ~k&=-\frac{1}{m_1}(x_1j+yk)e^{i\theta}+\frac{1}{m_2}e^{i\theta}(x_2j-yk),\\
&=-\frac{1}{m_1}x_1e^{-i\theta}j-\frac{1}{m_1}ye^{-i\theta}k+\frac{1}{m_2}x_2e^{i\theta}j-\frac{1}{m_2}ye^{i\theta}k.
\end{align*}
Multiplying both sides of this equation on the right by \(k\) gives an equation purely in terms of complex numbers. Separating this into real and imaginary parts results in a linear equation in \(x_1\) and \(x_2\) which may be written as
\begin{equation}\label{linear}
-\frac{1}{m_1m_2}\begin{pmatrix}
\sin\theta & \sin\theta \\ \cos\theta & -\cos\theta
\end{pmatrix}
\begin{pmatrix}
x_1 \\ x_2
\end{pmatrix}
=
\begin{pmatrix}
2|\eta|\sin\theta-y\left(\frac{1}{m_1}+\frac{1}{m_2}\right)\cos\theta \\
y\left(\frac{1}{m_1}-\frac{1}{m_2}\right)\sin\theta
\end{pmatrix}.
\end{equation}
This linear system is non-degenerate for \(\theta\ne\pi/2\) and has unique solutions 
\begin{equation}\label{x1x2}
x_1=y\left(\cot2\theta+\frac{m_1}{m_2}\csc2\theta \right)-m_1|\eta|,\quad\text{ and }\quad x_2=y\left(\cot2\theta+\frac{m_2}{m_1}\csc2\theta \right)-m_2|\eta|.
\end{equation}
\subsubsection*{Case 3: $\theta=\pi/2$}
For \(\theta=\pi/2\) the linear system in \eqref{linear} becomes degenerate. Solutions of the system can only exist when \(m_1\) and \(m_2\) are both equal to say \(m\). There is then an entire line's worth of solutions given by
\begin{equation}
\label{line_of_solns}
x_1+x_2=-2m^2|\eta|.
\end{equation}
\subsection{Reconstruction and the full relative equilibria classification}
In accordance with Proposition~\ref{REpropn}, having classified all RE solutions in \(M/G_L\), it remains to reconstruct the corresponding solutions in \(M\). Clearly, the action of \(G_L\) on \(M/G_L\) is trivial, and thus we may safely suppose that the corresponding one-parameter subgroup of \(G_L\) is generated by \(\xi=|\xi|j\in\mathfrak{g}_L\).

From the definitions, the real part of \(g_L\) and \(g_R\) must each be equal to \(r=\cos\theta\), where \(\theta\) is the angular separation between the two particles
\[
\cos\theta=\langle g_1,g_2\rangle=\langle g_1g_2^{-1},1\rangle=\langle 1,g_1^{-1}g_2\rangle=r.
\]
A similar set of equations to \eqref{one}, \eqref{two} and \eqref{three} also hold in the right reduced space, and it follows from these that \(\overline{g}_R\) must also be orthogonal to \(\xi\) and hence \(j\). Therefore, we may suppose that \(\overline{g}_R=-i\sin\theta\), and hence, that \(g_R=e^{-i\theta}\) and \(g_L=e^{i\theta}\). The definitions of \(g_L\) and \(g_R\) give \(g_2=g_1e^{i\theta}=e^{i\theta}g_1\). For when \(\theta\) is not equal to \(0\) or \(\pi\), this implies that \(g_1\) commutes with \(i\). This is only the case for when \(g_1\) belongs to the complex plane \(\C\subset\Q\). This forces \(g_1\) and \(g_2\) to be of the form
\begin{equation}
g_1=e^{-i\phi_1},\quad\text{ and }\quad g_2=e^{i\phi_2}
\end{equation}
where \(\phi_1+\phi_2=\theta\). We may additionally suppose that \(\phi_2\in[0,\pi]\), since \((g_1,g_2)\mapsto(-g_1,-g_2)\) is an element of \(SO(4)\). The orbit of a point \(q\) in \(\Q\) under the one-parameter subgroup of \(G_L\times G_R\) is \(e^{t\xi}qe^{-t\eta}\). For each particle \(g_i\) we can differentiate this motion to obtain the momentum \(p_i=m_i\dot{g}_i\), and then find the right momentum \(R_i=g_i^{-1}p_i\) as below
\begin{align}\label{rightmom}
R_1&=m_1(|\xi|\cos2\phi_1-|\eta|)j+m_1(|\xi|\sin2\phi_1)k,\\
R_2&=m_2(|\xi|\cos2\phi_2-|\eta|)j-m_2(|\xi|\sin2\phi_2)k\nonumber.
\end{align}
These expressions must agree with those in \eqref{R1R2form} for the forms of \(R_1\) and \(R_2\), and therefore the following must hold
\[
m_1\sin2\phi_1=m_2\sin2\phi_2=\frac{y}{|\xi|}.
\]
For when \(\theta\ne \pi/2\), this equation uniquely determines \(\phi_1\) and \(\phi_2\). For when \(\theta=\pi/2\), and therefore \(m_1=m_2\), the angles are not unique: for when \(y\) is positive the solutions are for all \(\phi_1\in(0,\pi/2)\), and all \(\phi_1\in(-\pi/2,0)\) for \(y\) negative. Finally, for the exceptional cases, we have \(g_1=g_2\) for \(\theta=0\), and \(g_1=-g_2\) for \(\theta=\pi\). As \(\theta\) is constant, it follows from consideration of Case 1 above, that these motions correspond to the two particles moving together around a great circle arbitrarily quickly, either occupying the same position, or antipodal to each other. 

\begin{thm}\label{main}
	For the 2-body problem on the 3-sphere with an either strictly attractive or repulsive potential, all relative equilibria solutions are completely classified, up to conjugacy, according to the angle \(\theta\in[0,\pi]\) subtended by both particles in the sense listed below. In each of these cases we suppose, without loss of generality, that the corresponding one-parameter subgroup of \(G_L\times G_R\) is generated by \((|\xi|j,|\eta|j)\in\mathfrak{g}_L\times\mathfrak{g}_R\).
	\begin{enumerate}
		\item[\textendash] For \(\theta=0\) and \(\pi\), we will call these solutions \emph{singular}. For \(\theta=0\) we may take the initial positions to be \(g_1=g_2=1\), and \(g_1=-g_2=1\) for \(\theta=\pi\). Both \(|\xi|\) and \(|\eta|\) are arbitrary. 
		\item[\textendash] For \(\theta<\pi/2\) we will call the solutions \emph{acute}, and \emph{obtuse} when \(\pi/2<\theta<\pi\). In both cases we may take the initial positions of the particles to be at \(g_1=e^{-i\phi_1}\) and \(g_2=e^{i\phi_2}\) where the angles \(\phi_1\) and \(\phi_2\) are determined by \(\phi_1+\phi_2=\theta\) and by
		\begin{equation}\label{little_lever}
		m_1\sin2\phi_1=m_2\sin2\phi_2.
		\end{equation}
		\item[\textendash] Solutions for \(\theta=\pi/2\) will be called \emph{right-angled} and only exist for \(m_1=m_2\). In this case, \(\phi_1\) and \(\phi_2\) are not uniquely determined by \eqref{little_lever} and may be any angles satisfying \(\phi_1+\phi_2=\pi/2\) where \(\phi_1\in(0,\pi/2)\) for a strictly attractive potential (\(f>0\)), or \(\phi_1\in(-\pi/2,0)\) for a strictly repulsive potential (\(f<0\)). In the case of equal masses, we will call the RE for any \(\theta\) \emph{isosceles} if \(\phi_1=\phi_2\).
	\end{enumerate}
	For all of the non-singular cases, the angular velocities satisfy
	\begin{equation}\label{lever}
	2|\xi||\eta|=f(\theta)\sin\theta/\zeta
	\end{equation}
	where \(\zeta=m_1\sin2\phi_1=m_2\sin2\phi_2\), and where \(f(\theta)=-dV/dr\) for \(V=V(1,e^{i\theta})\) and \(r=\cos\theta\).
\end{thm}
To understand these rigid motions, one can show that a one-parameter subgroup generated by \((|\xi|j,|\eta|j)\) in \(\mathfrak{g}_L\times\mathfrak{g}_R\) acts by rotating the oriented planes \(\Span\{1,j\}\) and \(\Span\{k,i\}\) through the angle \(|\xi|-|\eta|\) and \(|\xi|+|\eta|\) with each unit of time respectively. For when \(|\xi|=|\eta|\) this gives a simple rotation, and the two particles carry out a cospherical motion contained to the 2-sphere in \(\Span\{1,i,k\}\) rotated about the real line with angular velocity \(\omega=2|\xi|=2|\eta|\). In this case, the theorem above coincides with the RE classification given in \cite{james} for the situation on the 2-sphere.

We conclude this classification by remarking that, if we use the expressions in \eqref{x1x2} and \eqref{line_of_solns} for \(x_1\) and \(x_2\), one can express (although this is a fairly unremitting calculation) the explicit values of the generators of the invariant ring in \(\mathfrak{s}^*_+/G_R\) at the given RE. It is then possible to verify that the full reduced equations in \eqref{fulll_red_eqns} do indeed yield a fixed point for these values, as expected.
\subsection{Linearisation and the energy-momentum map}
Although the RE were classified by first finding the solutions in the left reduced space, the stability results will be derived by directly working in the full reduced space. As a RE is precisely a fixed point of the full reduced equations of motion, we may linearise the system in  \eqref{fulll_red_eqns} at such a point. Noting that \(R_1\) and \(R_2\) are orthogonal with \(\overline{g}_L\) at a RE, and thus, that \(k_{13}=k_{23}=0\), this linear system may be written as the following \(8\times 8\) matrix in \(\R^8\) with coordinates ordered as \((k_{11},k_{12},k_{13},k_{22},k_{23},k_{33},r,\delta)\),
\begin{equation}\label{8x8}
\begin{pmatrix}
0 & 0 & 2f & 0 & 0 & 0 & 0 & 0 \\
0 & 0 & -f & 0 & f & 0 & 0 & 0 \\
-r/m_1 & r/m_2 & 2f & 0 & 0 & f & A_1 & -1/m_2 \\
0 & 0 & 0 & 0 & -2f & 0 & 0 & 0 \\
0 & -r/m_1 & 0 & r/m_2 & 0 & -f & A_2 & 1/m_1 \\
0 & 0 & -2r/m_1 & 0 & 2r/m_2 & 0 & 0 & 0 \\
0 & 0 & 1/m_1 & 0 & -1/m_2 & 0 & 0 & 0 \\
0 & 0 & A_3 & 0 & A_4 & 0 & 0 & 0 
\end{pmatrix}.
\end{equation}
We have used the abbreviations
\begin{align*}
A_1&=+\frac{df}{dr}k_{33}-\left(\frac{k_{11}}{m_1}-\frac{k_{12}}{m_2}\right),\\
A_2&=-\frac{df}{dr}k_{33}-\left(\frac{k_{12}}{m_1}-\frac{k_{22}}{m_2}\right),\\
A_3&=+\frac{k_{12}}{m_1}+\frac{k_{22}}{m_2},\\
A_4&=-\frac{k_{11}}{m_1}-\frac{k_{12}}{m_2}.
\end{align*}
As one might expect, the characteristic polynomial of such a large symbolic matrix is fairly horrendous. For this reason, we are forced to make a choice for the force \(f\). The two potentials we will consider are the \emph{gravitational potential}
\begin{equation}\label{2body_pot}
V=-m_1m_2\cot\theta\quad\text{ with }\quad f=m_1m_2\csc^3\theta
\end{equation}
for the planetary 2-body problem (see \cite{diacu} for a justification of this choice), and 
\begin{equation}\label{lagrange_pot}
V=\gamma\cos\theta\quad\text{ with }\quad f=-\gamma
\end{equation}
for the potential of the 4-dimensional Lagrange top as in \eqref{Lagrange_ham}. With these choices of \(f\) computing the characteristic polynomial becomes feasible, although we have enjoyed the aid of software which can handle such symbolic calculations. Before presenting these polynomials, we remark that we should expect at least four of the eigenvalues to be zero. This is because the generic symplectic leaves are 4-dimensional, and the fixed points of the flow should vary continuously as we move through the leaves.

For the 2-body potential the characteristic polynomial in the variable \(t\) is given by
\begin{equation}\label{2bodycharpoly}
t^4(c_0+c_2t^2+t^4)
\end{equation}
where
\[
c_2=2 \left(\frac{k_{11}}{m_1^2}+\frac{k_{22}}{m_2^2}+\left(m_1+m_2\right) \cot\theta\csc ^2\theta\right)
\]
and
\[
c_0=\left(\frac{k_{11}}{m_1^2}-\frac{k_{22}}{m_2^2}\right)^2+2\cot\theta\csc^2\theta\left[\frac{k_{11}}{m_1}\left(1+\frac{m_2}{m_1}\right)+\frac{k_{22}}{m_2}\left(1+\frac{m_1}{m_2}\right)\right]+\left[(m_1+m_2)\cot\theta\csc^2\theta\right]^2.
\]

Before proceeding to give the characteristic polynomial for the Lagrange top, we note that it follows from \eqref{x1x2} that \(|R_1|=|R_2|\) for when \(\theta\ne\pi/2\). We will therefore first give the polynomial for when \(\theta\ne\pi/2\) and where we set \(k_{11}=k_{22}=|R|^2\). This characteristic polynomial is then 
\begin{equation}\label{lagrangepoly}
t^4 \left(t^2-2 \alpha \gamma \cos (\theta )\right) \left(t^2+4 \alpha^2 |R|^2-8 \alpha \gamma \cos (\theta )\right).
\end{equation}
For when \(\theta=\pi/2\) it is no longer the case that \(k_{11}=k_{22}\). The resulting polynomial is then
\begin{equation}\label{rightanglepoly}
t^4\left(t^4+2\alpha^2(k_{11}+k_{22})t^2+\alpha^4(k_{11}-k_{22})^2\right).
\end{equation}

In the next two subsections we analyse the roots of these characteristic polynomials accompanied by images of the energy-momentum map. The RE are equivalently defined to be critical points of this map, and the set of critical values helps illuminate the study of the RE by understanding how they bifurcate and change in nature.

We have opted to omit the full details of how the images of the energy-momentum (actually the energy-Casimir) map are obtained. Ultimately this task is pure calculation, but it does deserve some comments on how to obtain it in practice. Following Theorem~\ref{main} one can parameterise the families of RE by \(\theta\) and \(|\eta|\) alone, or \(\phi_1\) and \(|\eta|\) for the right-angled RE. One can then acquire explicit formulae for \(H\), \(|\lambda|^2\) and \(|\rho|^2\) in terms of these. This statement conceals an implicit exercise in elementary geometry to express \(\zeta\) and relate the angles \(\phi_1\), \(\phi_2\) to \(\theta\) in \eqref{lever}. The energy-Casimir map simplifies considerably upon introducing the reparameterisation
\begin{equation}\label{tau}
2e^\tau|\eta|^2=f\sin\theta/\zeta.
\end{equation}
For our purposes it will be enough to state that the image of each family of RE under the energy-Casimir map can be parameterised by \(\theta\) and \(\tau\) alone, or \(\phi_1\) and \(\tau\) for the right-angled RE.

\subsection{Stability of relative equilibria for the 2-body problem}
\begin{figure}
	\centering
	\begin{subfigure}[b]{0.4\textwidth}
		\includegraphics[width=\textwidth]{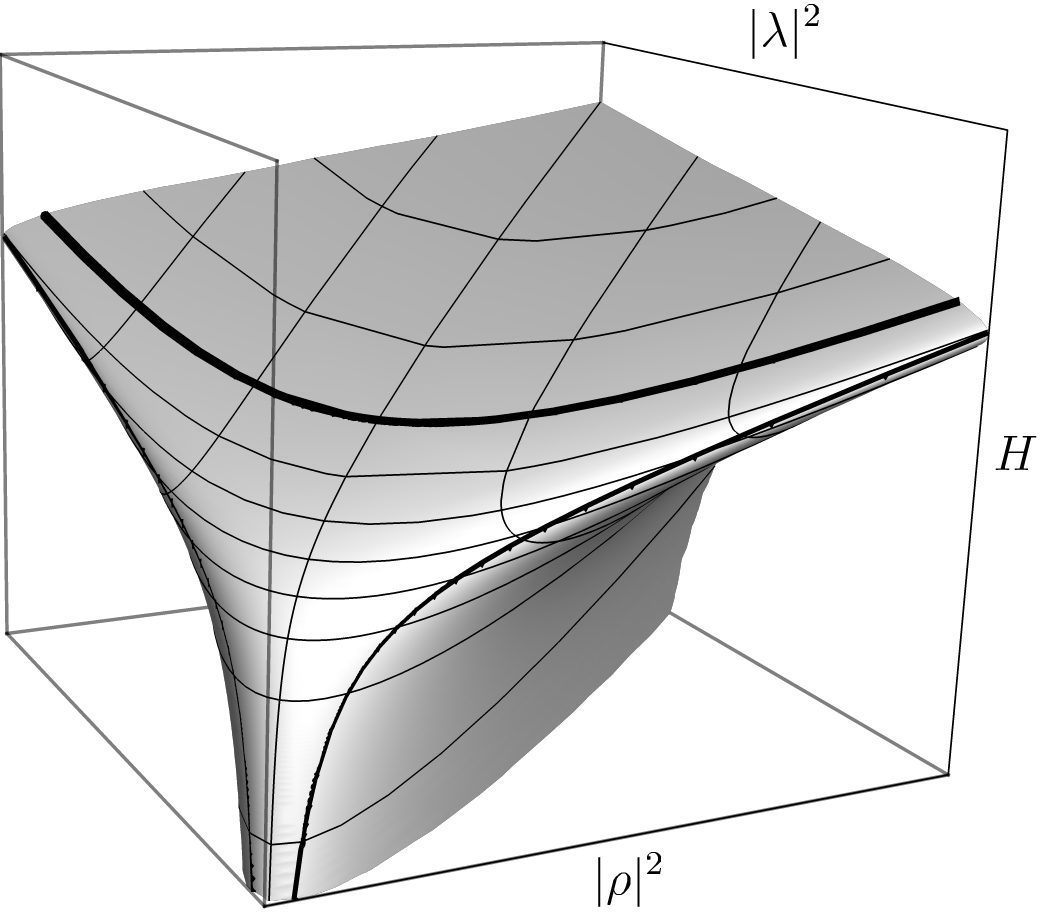}
		\caption{$0<\theta<\pi$, isosceles}
		
	\end{subfigure}
	\qquad\quad
	\begin{subfigure}[b]{0.4\textwidth}
		\includegraphics[width=\textwidth]{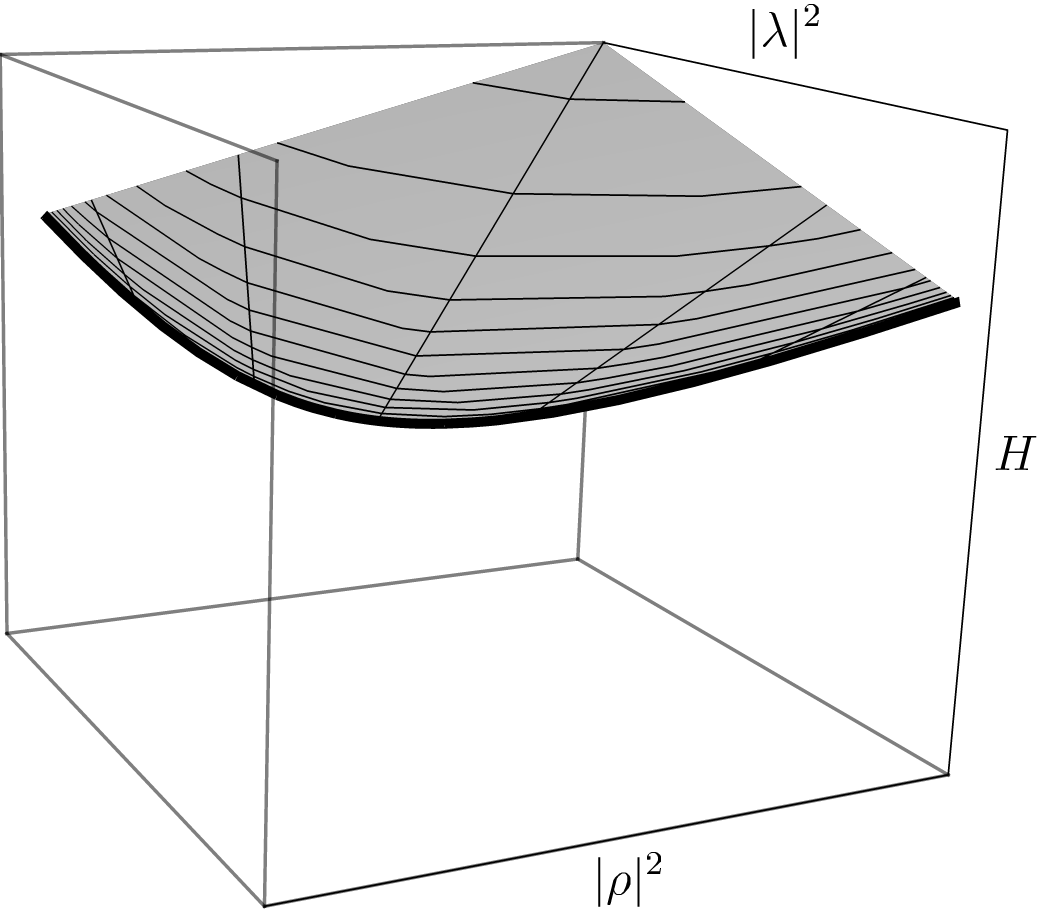}
		\caption{$\theta=\pi/2$, right-angled RE}
		
	\end{subfigure}
	\caption{\label{equal}The energy-Casimir bifurcation diagram for the case of equal masses. The axes for both diagrams are to the same scale. For the isosceles component, the coordinate lines running from left to right are of constant \(\theta\), and those transversal to them are for constant \(\tau\). The line \(\theta=\pi/2\) is thickened, and it is along this line that the component for the right-angled RE is attached.}
\end{figure}
Surprisingly, despite the complicated appearance of the coefficients in \eqref{2bodycharpoly}, all four roots may be solved and compactly written as 
\begin{align}
z_{1,2}&=\pm\sqrt{-\left(\frac{\sqrt{k_{11}}}{m_1}+\frac{\sqrt{k_{22}}}{m_2}\right)^2-(m_1+m_2)\cot\theta\csc^2\theta},\\
w_{1,2}&=\pm\sqrt{-\left(\frac{\sqrt{k_{11}}}{m_1}-\frac{\sqrt{k_{22}}}{m_2}\right)^2-(m_1+m_2)\cot\theta\csc^2\theta}.
\end{align}
It is clear that these eigenvalues are all purely imaginary when \(\theta\) is acute. Furthermore, by writing \(k_{11}\) and \(k_{22}\) in terms of \(\theta\) and \(|\eta|\) using the forms in \eqref{R1R2form} and \eqref{x1x2}, one can show that
\[
\frac{k_{11}}{m_1^2}+\frac{k_{22}}{m_2^2}+(m_1+m_2)\cot\theta\csc\theta=\frac{1}{8|\eta|^2}\left(16|\eta|^4\cos^2\theta\sin^6\theta+(m_1^2+m_2^2+2m_1m_2\cos2\theta)\right).
\]
This expression is always greater than zero, and so the eigenvalue pair \(z_{1,2}\) is always purely imaginary and non-zero. On the other hand, the eigenvalue pair \(w_{1,2}\) does undergo a transition from imaginary to real. For when the masses are equal, it follows from \eqref{x1x2} that \(k_{11}=k_{22}\) for the isosceles RE, and therefore, that \(w_{1,2}\) is a non-zero real pair for \(\theta\) obtuse, and zero for \(\theta=\pi/2\). For the remaining right-angled RE which are not isosceles, as \(k_{11}\ne k_{22}\) the \(w_{1,2}\) roots are a non-zero imaginary pair. 

For non-equal masses it becomes more difficult to describe the transition in reality of the \(w_{1,2}\) pair. Unlike the case for motion on the 2-sphere in \cite{james}, this transitions is not determined solely by a critical angle. With reference to the energy-Casimir diagram in Figure~\ref{obtuse}, we argue that this transition occurs along the \emph{fold} in the obtuse component. A rigorous proof of this requires extremely lengthy calculations and so we will merely sketch it here. The image of the energy-Casimir map is given as a surface parameterised by \(\theta\) and \(\tau\). As this surface is folded, the curves of constant \(|\lambda|^2\) and \(|\rho|^2\) in \((\theta,\tau)\)-space generically intersect in two points. Along the fold where such a pair of \(|\lambda|^2\) and \(|\rho|^2\) occurs only once, these two curves must intersect tangentially at a point. The fold may then be characterised by the condition that the Jacobian of \(|\lambda|^2\) and \(|\rho|^2\) with respect to \(\theta\) and \(\tau\) vanishes. This condition turns out to give a quadratic in \(\cosh\tau\). As the pair \(z_{1,2}\) is always imaginary, the reality of \(w_{1,2}\) is determined by the sign of \(c_0\) in \eqref{2bodycharpoly}: the pair is imaginary when \(c_0\ge0\), and real for \(c\le 0\). The expression for \(c_0\) may be written in terms of \(\theta\) and \(\tau\) using \eqref{x1x2} and the reparameterisation in \eqref{tau}. Setting \(c_0\) to zero then gives an expression for \(\cosh\tau\) in terms of \(\theta\), and upon substituting this into the Jacobian, yields zero as desired. It is the portion of the surface above the fold for which \(w_{1,2}\) is imaginary, and after crossing the fold, the portion below for which \(w_{1,2}\) is real.

Now that we have a picture for the eigenvalues of the linearisation at the RE, we have sufficient conditions for instability but not stability. A helpful result in this direction would be the signature of the Hessian for the hamiltonian at the RE restricted to a leaf. A definite Hessian implies Lyapunov stability, however we should be pessimistic about this prospect since the Hessian obtained in \cite{james} for the linearly stable RE on the 2-sphere have mixed signature. In this situation, stability is again ambiguous. We encourage the reader to observe that the energy-momentum diagrams given in \cite{james} coincide with our figures for when \(|\lambda|=|\rho|\). This is not a coincidence, but instead a consequence of Sjamaar's Principle \cite{sjamaar} which, roughly speaking, establishes a symplectomorphism between the reduced spaces with the reduced spaces of isotropy submanifolds equipped with the isotropy action. We can be more explicit about the implementation of this principle for our example.
\begin{propn}
	Let the 2-sphere be given by the unit imaginary quaternions \(S^2\subset\Imag\Q\) and consider the action of \(G\) on \(N=T^*S^2\times T^*S^2\) with momentum map \(J\). For any \(\Omega\in\mathfrak{g}^*\) the reduced space \(N_\Omega\) is symplectomorphic to \(M_{\Omega,-\Omega}\).
\end{propn}
\begin{proof}
	For \((g_1,p_1,g_2,p_2)\) cospherical in \(N\) the value of \(J\) is the angular momentum
	\[
	\Omega=(g_1\times p_1)+(g_2\times p_2).
	\]
	As these quaternions are imaginary, we have \(g_i^{-1}=-g_i\). By expanding the cross product as multiplication between quaternions and taking the complex conjugate of the \(R_i\), one can show that \(2\Omega=\lambda-\rho\) and \(\lambda=-\rho\). As the isotropy subgroup fixing a 2-sphere in \(\Q\) is trivial, the symplectic submanifold \(N\hookrightarrow M\) descends to the orbit spaces to give an inclusion \(N/G\hookrightarrow M/(G_L\times G_R)\). From Proposition~\ref{crit}, any point with \(|\lambda|=|\rho|\) is cospherical, and thus, at the level of orbit spaces there is a bijection between the orbit reduced spaces \(J^{-1}(\orb_\Omega)/G\rightarrow J_{L,R}^{-1}(\orb_\Omega\times\orb_{-\Omega})/G_L\times G_R\) induced by the canonical inclusion. As all of the maps here are Poisson, it follows that this bijection between symplectic leaves is a symplectomorphism.
\end{proof}
One sees from this proof that the reduced hamiltonian on \(M_{\Omega,-\Omega}\) pulls back to give the hamiltonian on \(N_\Omega\). It follows that the Hessian at the RE must be the same, and therefore, that the Hessian of the RE in the energy-Casimir diagrams in Figures~\ref{equal} and \ref{obtuse} agrees with that in \cite{james} for when \(|\lambda|=|\rho|\). As the Hessian varies continuously with the relative equilibria, and that the signature can only change when it crosses a zero eigenvalue, we may apply a continuity argument to extend the Hessian over all of the RE with non-zero eigenvalues. This observation, combined with the above discussion concerning the linearisation provides the proof to the following theorem.
\begin{thm}
	For the 2-body problem on the 3-sphere with the gravitational potential, we have the following stability results for the RE:
	\begin{enumerate}
			\item[\textendash] All acute RE are linearly stable, with the signature of the Hessian of the hamiltonian being \((+,+,-,-)\).
			\item[\textendash] For when the masses are equal, all right-angled RE which are not isosceles are linearly stable with signature \((+,+,-,-)\).
			\item[\textendash] Obtuse RE which are above the fold in Figure~\ref{obtuse} are linearly stable with signature \((+,+,-,-)\). Obtuse RE under the fold are linearly unstable with signature \((+,+,+,-)\). For when the masses are equal there is no such fold, and all obtuse RE are linearly unstable with signature \((+,+,+,-)\).
	\end{enumerate}
\end{thm}
\begin{figure}
	\centering
	\begin{subfigure}[b]{0.4\textwidth}
		\includegraphics[width=\textwidth]{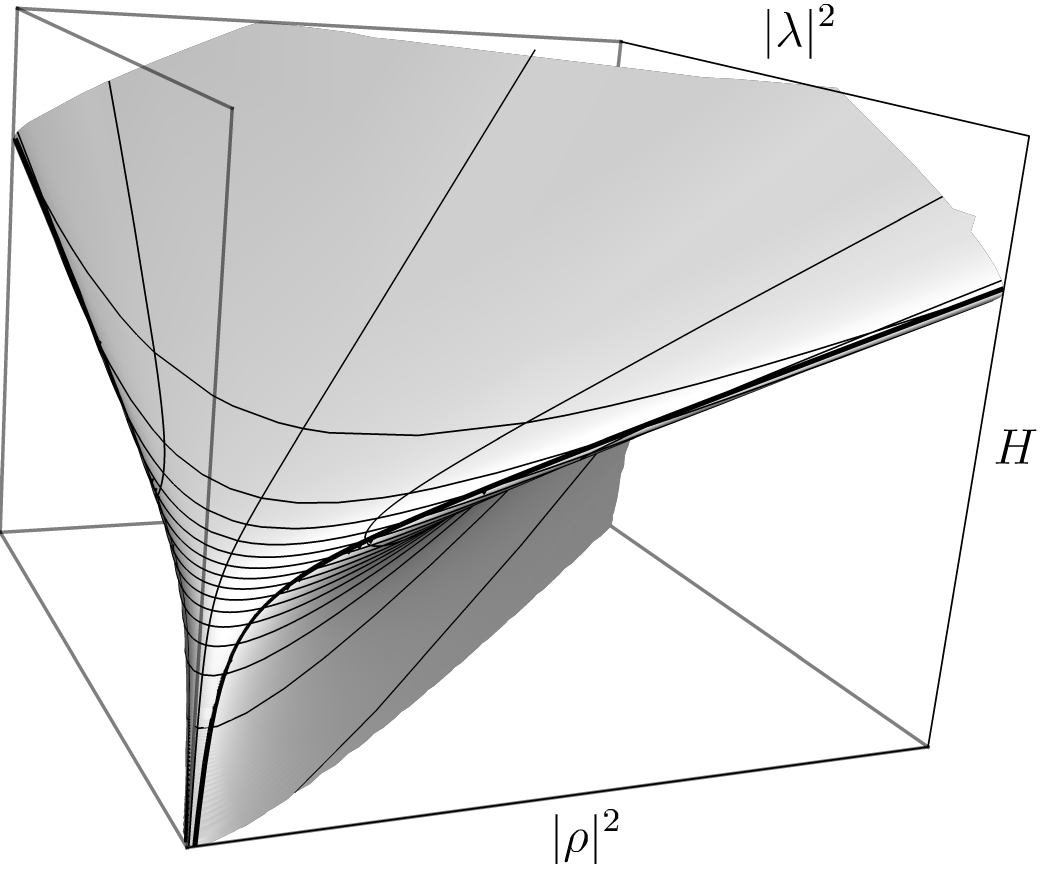}
		\caption{$0<\theta<\pi/2$, acute}
		
	\end{subfigure}
	\qquad\quad
	\begin{subfigure}[b]{0.4\textwidth}
		\includegraphics[width=\textwidth]{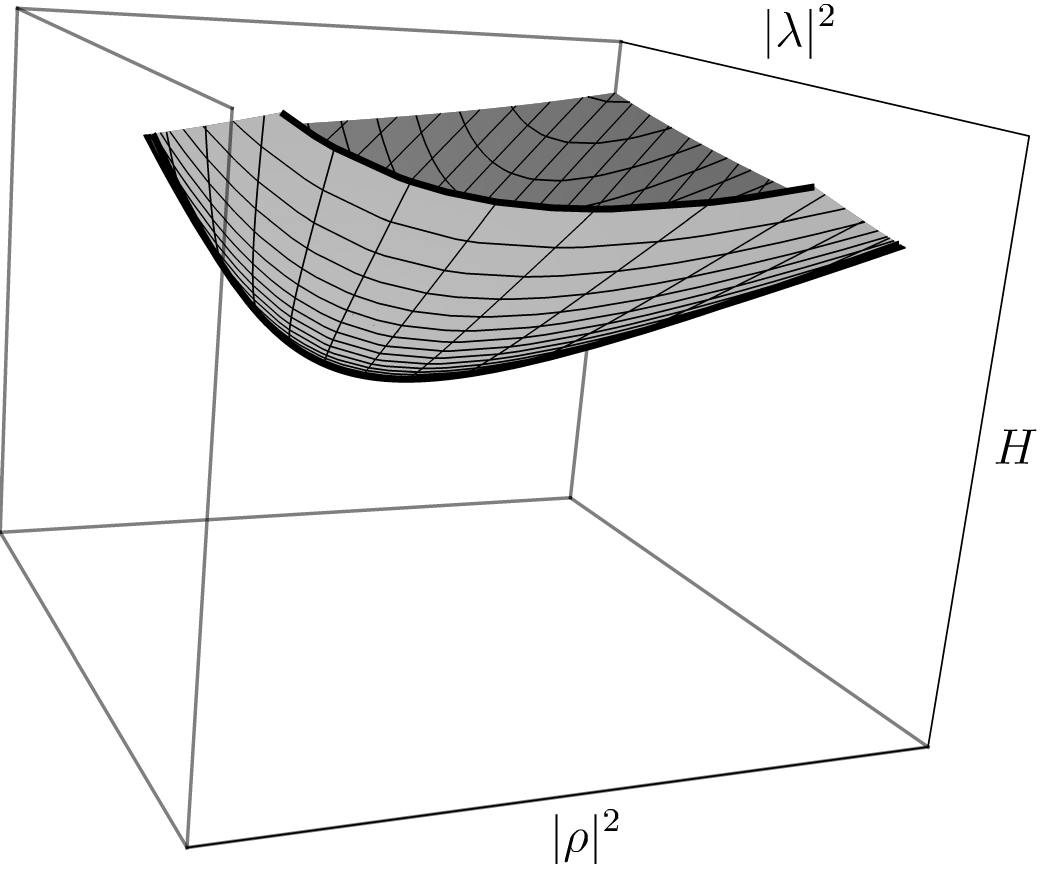}
		\caption{$\pi/2<\theta<\pi$, obtuse}
	
	\end{subfigure}
	\caption{\label{obtuse}The energy-Casimir bifurcation diagram for non-equal masses, specifically for $m_1=3$ and $m_2=2$. The axes for both diagrams are to the same scale. The coordinate lines running from left to right are of constant \(\theta\), and those transversal to them are for constant \(\tau\). We have deliberately removed an upper section of the obtuse component to demonstrate that the surface is folded along a cusp. The lines of constant \(\theta\) above the fold are for \(\theta\) smaller than those lines below.}
\end{figure}
\subsection{Stability of the relative equilibria for the Lagrange top}
Before jumping straight into a stability analysis, we pause to review some of our terminology for the Lagrange top RE. As the angle \(\theta\) corresponds to the angle the body axis makes with the vertical, we will alternatively refer to acute and obtuse RE by \emph{upright} and \emph{downward} respectively. Furthermore, right-angled RE with \(\theta=\pi/2\) might more fittingly be described as being \emph{horizontal}. Alternatively of course, the potential in \eqref{lagrange_pot} applies equally well to the case of two particles on a sphere which are repelled by each other. In this case, the old terminology is perfectly applicable.

From the characteristic polynomial in \eqref{lagrangepoly}, one can immediately see from the root \(t^2=2\alpha\gamma\cos\theta\) that RE which are upright are unstable. This is in marked difference to the ordinary Lagrange top where a sufficiently quickly spinning top is stable when standing vertical (the sleeping top). We deduced this result earlier in Figure~\ref{tube} from the flow on the singular strata. The second pair of roots may be analysed by again using \eqref{x1x2} to write \(|R|^2=|R_1|^2=|R_2|^2\) in terms of \(\theta\) and \(|\eta|\) which gives
\begin{equation}\label{spin}
4 \alpha^2 |R|^2-8 \alpha \gamma \cos \theta=4|\eta|^2+\frac{\alpha^2\gamma^2}{|\eta|^2}-4\alpha\gamma\cos\theta.
\end{equation}
This is always strictly positive for \(\theta\) greater than zero, and therefore always corresponds to imaginary eigenvalues. For the right-angled RE we have the characteristic polynomial in \eqref{rightanglepoly}, for which the roots of the quartic factor are equal to
\begin{equation}
\pm\sqrt{-\alpha^2(\sqrt{k_{11}}\pm\sqrt{k_{22}})^2}.
\end{equation}
Consequently, we see that these eigenvalues are all non-zero and imaginary away from the isosceles \(k_{11}=k_{22}\) branch.

Unlike for the 2-body problem, where we ignored the singular RE since the hamiltonian is not defined for \(\theta=0\) or \(\pi\), the energy-Casimir diagram for the Lagrange top contains two `threads' corresponding to these families of RE. From Figure~\ref{lagrange} for the `isosceles' RE, one sees that the lines of constant \(\theta\) converge to a single thread on top of the surface as \(\theta\rightarrow 0\). This thread corresponds to those configurations of the top held vertically upright and spinning about its axis. As the spin decreases, that is, as \(k_{11}=k_{22}=|R|^2\) decreases below \(2\gamma/\alpha\), the second root pair in \eqref{lagrangepoly} transitions from imaginary to real, and the thread detaches from the surface and extends outwards as an isolated thread until \(|R|=0\), where the top is motionless. This is entirely analogous to the case for the ordinary Lagrange top as shown in \cite{cushbook} and is a mathematical realisation of `gyroscopic stabilisation'. It cannot be seen in our picture, however as \(\theta\) tends to \(\pi\), the lines of constant \(\theta\) converge underneath the surface to a thread which is not isolated, corresponding to those motions for when the body is hanging vertically downwards. 

As with the 2-body problem, we remark that taking the slice through the energy-Casimir diagram for \(|\rho|=|\lambda|\) in Figure~\ref{lagrange} gives the same diagram as in \cite{borisov2004}, where they also consider the 2-body problem on the 2-sphere for the same potential.

\begin{figure}
	\centering
	\begin{subfigure}[b]{0.4\textwidth}
		\includegraphics[width=\textwidth]{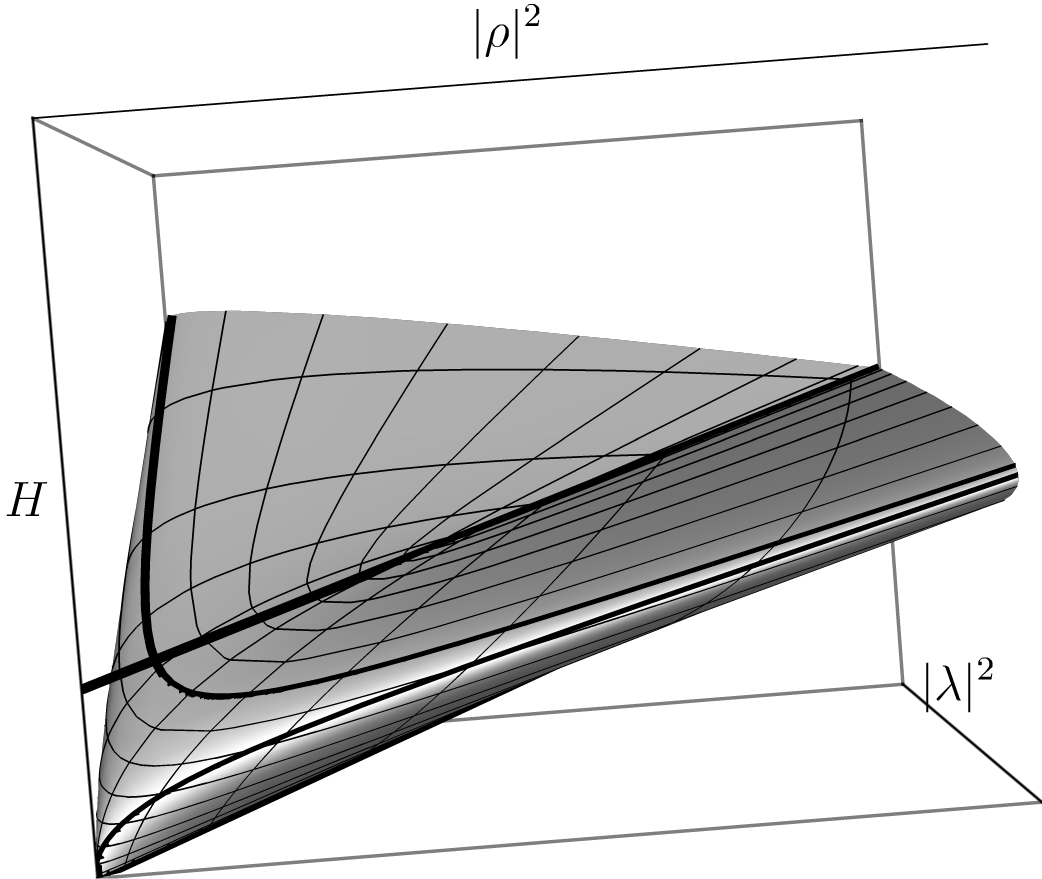}
		\caption{$0<\theta<\pi$, `isosceles'}
		
	\end{subfigure}
	\qquad\quad
	\begin{subfigure}[b]{0.4\textwidth}
		\includegraphics[width=\textwidth]{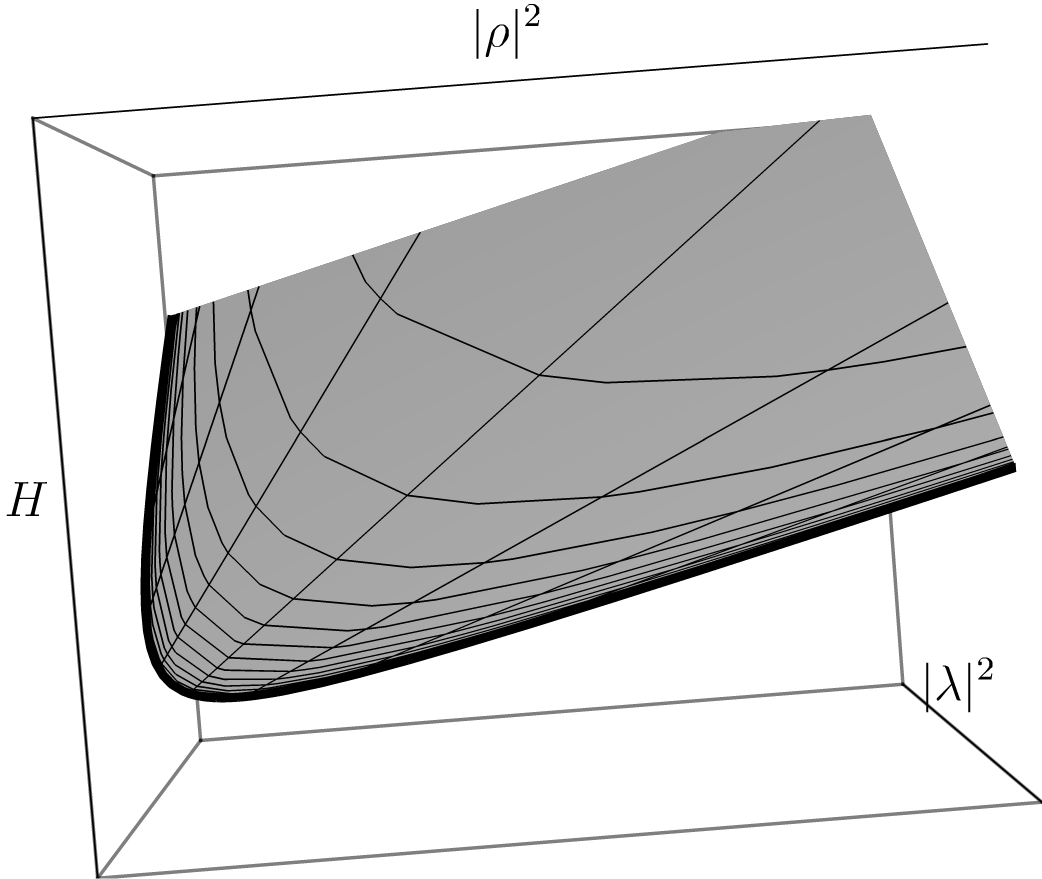}
		\caption{$\theta=\pi/2$, `right-angled'/horizontal RE}
		
	\end{subfigure}
	\caption{\label{lagrange}The energy-Casimir bifurcation diagram for the 4-dimensional Lagrange top, specifically for \(\alpha=2\). The axes for both diagrams are to the same scale. For the component on the left, the coordinate lines emanating away from the origin are of constant \(\tau\), and those transversal to them are of constant \(\theta\). 
		The thickened line is that for \(\theta=\pi/2\), and it is along this branch that the additional component of horizontal RE is attached.}
\end{figure}

A linearised stability analysis is sufficient for deducing instability, but not conclusive for stability. Fortunately, the existence of the additional integral \(I\) in \eqref{INTEGRAL} on the full reduced space will allow us to obtain the strongest possible results for stability. We begin by claiming that all downward RE are not only critical points of the hamiltonian, but also of \(I\). To see this, observe that for \(\theta>\pi/2\), the linearisation admits two distinct non-zero imaginary eigenvalue pairs. If such a point were not a critical point of \(I\), then in a neighbourhood of this point we could introduce coordinates which include \(I\) as a coordinate function. As \(\{H,I\}=0\), the hamiltonian in these coordinates is independent of \(I\), and therefore a change in the \(I\) coordinate away from a RE would also give a RE. This is not compatible with the non-zero eigenvalues.

Therefore, at such RE the differentials \(dH\) and \(dI\) are both zero. We also claim that their Hessians \(d^2H\) and \(d^2I\) are linearly independent. This is equivalent to showing that the linearisations of the flows generated by \(H\) and \(I\) at the fixed point are independent. With help from Table~\ref{table} one can find the flow generated by the integral \(I\). In particular, one can show 
\begin{align*}
\dot{k}_{11}&=\{I,k_{11}\}=4\gamma(k_{13}k_{12}-k_{11}k_{23}),\\
\dot{k}_{22}&=\{I,k_{22}\}=4\gamma(k_{13}k_{22}-k_{12}k_{23}).
\end{align*}
By linearising these at a RE, and comparing with the first and fourth rows of the matrix in \eqref{8x8}, we see that the two linearisations are indeed linearly independent for when \(k_{11},k_{12},k_{22}\ne 0\), and therefore so too are \(d^2H\) and \(d^2I\).

The quadratic forms \(d^2H\) and \(d^2I\) are both well defined on the tangent space of the symplectic leaf at a RE since it is a critical point for each of them. The Lie algebra of such quadratic forms with respect to the Poisson bracket is isomorphic to the symplectic Lie algebra \(\text{Symp}(4;\R)\). Furthermore, as \(\{H,I\}=0\), the quadratic forms also commute, and as \(d^2H\) has distinct eigenvalues, and is linearly independent from \(d^2I\), it follows that they span a Cartan subalgebra of \(\text{Symp}(4;\R)\). Up to conjugacy by canonical transformations there are only four such Cartan subalgebras: center-center, saddle-centre, saddle-saddle, and focus-focus \cite{bolsinov}. As the eigenvalues of the linearisation are all purely imaginary, this forces it to be of center-center type. It follows from a normal-form result in \cite{lerman} that there exist Darboux coordinates \((q_1,p_1,q_2,p_2)\) in a neighbourhood of the RE (which may be taken to be the origin) where 
\begin{align*}
H&=a(q_1^2+p_1^2)+b(q_2^2+p_2^2)+\dots\\
I&=c(q_1^2+p_1^2)+d(q_2^2+p_2^2)+\dots
\end{align*}
Here the dots denote terms of cubic order in the coordinates, and where we have further supposed that \(I\) and \(H\) are zero at the origin. As the quadratic forms are linearly independent, we can find \(x,y\in\R\) such that the function \(F=xH+yI\) has \(d^2F\) positive definite at the RE. The function \(F\) therefore has a minimum at this point and thus, for any small \(\delta>0\) there exists an \(\varepsilon\) for which \(F^{-1}(\epsilon)\) is contained to a ball of radius \(\delta\). It follows then from
\[
 H^{-1}(\varepsilon_1)\cap I^{-1}(\varepsilon_2)\subset F^{-1}(x\varepsilon_1+y\varepsilon_2)
\]
that the level sets of \(H\) and \(I\) are also contained to arbitrarily small neighbourhoods around the RE. As the flow is contained to these level sets, small perturbations away from these RE result in motions contained to tori which remain close to the RE.

 To be completely watertight, the single RE which has evaded our argument so far is the one corresponding to the body hanging vertically downward and motionless as \(k_{11}=k_{12}=k_{22}=0\), and thus the two Hessians may not be independent. For this point, Lyapunov stability is a consequence of it being a global minimum of the hamiltonian. We can now pull all of this together into a theorem.
\begin{thm}
	For the 4-dimensional Lagrange top the following stability results apply to the relative equilibria in the full reduced space. 
	\begin{enumerate}
		\item[\textendash]All {upright} relative equilibria, that is, those with \(\theta<\pi/2\), are linearly unstable. 
		\item[\textendash] All {hanging} relative equilibria with \(\theta>\pi/2\), and all {horizontal} relative equilibria with \(\theta=\pi/2\), excluding those which are isosceles, are Lyapunov stable. 
	\end{enumerate}
\end{thm}

\section*{Concluding comments and the scope for further work}
It is natural to ask how we might extend these results. In this regard it is crucial to note that the fundamental idea upon which this work rests is the `accidental' isomorphism between \(\mathfrak{g}\times\mathfrak{g}\) and \(\mathfrak{so}(4)\). It is thanks to this that we have the double cover over \(SO(4)\) and the connection with the Lagrange top, and the commuting left and right actions which allow us to reduce in stages. It is because of this `accident' that our work does not generalise to more bodies, or to the negative curvature case of the 2-body problem on hyperbolic 3-space. For this space, the symmetry group \(SO(1,3)\) is double covered by \(SL(2;\C)\). It is clear that a different approach must be taken to resolve this case.

We do however obtain a fairly straightforward generalisation of our work by replacing the algebra of real quaternions with the split quaternions. This alteration results in the 2-body problem on \(SO(1,2)\) which is the unit `sphere' in this algebra. Everything then proceeds almost exactly the same, the full symmetry group is \(SO(2,2)\) and this is double covered by two `spheres' and we have again two commuting left and right actions. This alteration is entirely due to another `accidental' isomorphism between \(\mathfrak{so}(1,2)\times\mathfrak{so}(1,2)\) and \(\mathfrak{so}(2,2)\). In a similar fashion we could even push this idea further and replace the quaternions with the biquaternions and consider the 2-body problem on \(SL(2;\C)\).

Another route of study concerns the more general problem of dynamics on a cotangent bundle of a group \(G\) which is symmetric with respect to both the left and right translations by a given subgroup \(H\). In addition to the example we have dealt with, a famous example of such a system are the Riemannian ellipsoids \cite{riemann}. This system concerns the motion of a self-gravitating distribution of mass whose configuration is given by an element of \(G=SL(3;\R)\) which is symmetric by the left and right actions of \(H=SO(3)\). This system is comprehensively treated in the work of Chandrasekhar in \cite{chand}, and a more modern hamiltonian account may be found in \cite{modern}. The famous work of Riemann concerns the classification of the relative equilibria, and one wonders whether our use of reduction by stages could be applied as a possible alternative approach.

It would be interesting to study the limit as one particle's mass dominates the other. This should be expected to approach what is referred to as the restricted 2-body problem on the sphere \cite{restricted}. In particular, it would be interesting to see what the flow on the full reduced space limits to. Furthermore, given that we have the Poisson structure on the full reduced space, it would be nice to see if this offers any use in demonstrating the non-integrability for the 2-body problem (see \cite{nonintegrability}) or whether additional integrable systems can be found for different potentials. One might hope that this would connect with the substantial literature that exists for integrable systems on \(SO(4)\).

Finally, our work concerning the 2-body problem on the sphere cannot be considered complete, and there remain interesting unresolved questions. The nature of the full reduced spaces is one such question. One would like to say more about their geometry, to describe the fibres of the energy-momentum map, and the invariant integral manifolds. Moreover, as the Lagrange top is integrable, one can also ask questions about the foliation of the reduced spaces into invariant tori, the image of the momentum map, and the monodromy of these tori. The stability of the RE for the 2-body problem also remains an open question. As in \cite{james}, we leave the door open for the use of sophisticated KAM methods to strengthen the stability results.
\subsection*{Acknowledgements}
I would like to extend special thanks to my supervisor James Montaldi, for his invaluable teachings over the past few years, and for his helpful thoughts and discussions offered in the preparation of this paper. Whenever I have walked into his office he has always given me his time and careful attention, and for this I am extremely grateful.
\bibliographystyle{alpha}
\bibliography{1234} 
\newcommand{\Addresses}{{
		\bigskip
		\footnotesize
		
		P.~Arathoon, \textsc{School of Mathematics, University of Manchester,
			Manchester, M13 9PL, U.K.}\par\nopagebreak
		\textit{E-mail address},  \texttt{philip.arathoon@manchester.ac.uk}
		
		%
		%
		%
		
}}

\Addresses
\end{document}